\documentclass[12pt]{amsart}
\usepackage{amsfonts,amsmath,amssymb}
\usepackage[mathscr]{euscript}
\usepackage{graphicx}
\usepackage{indentfirst}
\usepackage{hyperref}
\usepackage{mathtools}
\usepackage{subcaption}
\usepackage{enumitem}
\usepackage{etoolbox}
\usepackage{hyperref}
\usepackage{needspace}
\usepackage{tikz}

\usetikzlibrary{arrows,decorations.markings,intersections}
\usetikzlibrary{shapes.misc}
\usetikzlibrary{patterns}
\tikzset{cross/.style={cross out, draw=black, minimum size=2*(#1-\pgflinewidth), inner sep=0pt, outer sep=0pt},
  cross/.default={2pt}}

\newtheorem{thm}{Theorem}[section]

\newtheorem{cor}[thm]{Corollary}
\newtheorem{lem}[thm]{Lemma}
\newtheorem{prop}[thm]{Proposition}
\newtheorem*{prob*}{Problem}
\newtheorem*{thm*}{Theorem}

\theoremstyle{definition}

\newtheorem*{defn*}{Definition}
\newtheorem{rem}[thm]{Remark}

\newtheorem*{rem*}{Remark}
\numberwithin{equation}{section}

\newcommand{\C}{\mathbb C}

\newcommand{\R}{\mathbb R}

\DeclareMathOperator{\vp}{v.p.}

\DeclareMathOperator{\crit}{crit}

\DeclareMathOperator{\Res}{Res}

\newcommand{\re}{\mathop{\mathrm{Re}}}
\newcommand{\im}{\mathop{\mathrm{Im}}}

\newcommand{\Tr}[1]{\mathrm{Tr}\, #1}
\newcommand{\op}[1]{\mathrm{#1}}
\newcommand{\fdet}[2][]{{\mathrm{det}}_{#1}\!\left(#2\right)}
\newcommand{\mexp}[2][]{\mathbb E_{#1}\! \left[#2\right]}
\newcommand{\prob}[2][]{\mathbb P_{#1}\! \left\{#2\right\}}
\newcommand{\myexp}[1]{\exp{\!\left( {#1} \right)}}

\newtheoremstyle{prbl}
{}                                   
{}                                   
{\itshape}                           
{}                                   
{\bfseries}                          
{}                                   
{ }                                  
{\thmname{#1}\thmnote{ #3}}          
\theoremstyle{prbl}

\newtheorem{problem}{Problem}

\makeatletter
\def\namedlabel#1#2{\begingroup
  \def\@currentlabel{#2}%
  \label{#1}\endgroup
}

\begin{document}
\title[Gap probabilities for products of random matrices]
{\bf{Gap probability for products of random matrices\\ in the critical regime}}
\author{Sergey Berezin}
\address{Department of Mathematics, The Hebrew University of Jerusalem, Givat Ram, Jerusalem 91904, Israel; V.A. Steklov Mathematical Institute of RAS, 27 Fontanka 191023, St. Petersburg, Russia.}\email{servberezin@yandex.ru, serberezin@math.huji.ac.il}
\author{Eugene Strahov}
\address{Department of Mathematics, The Hebrew University of Jerusalem, Givat Ram, Jerusalem 91904, Israel}
\email{strahov@math.huji.ac.il}
\commby{}
\begin{abstract}
  The singular values of a product of~$M$ independent Ginibre matrices of size~$N\times N$ form a determinantal point process. Near the soft edge, as both~$M$ and~$N$ go to infinity in such a way that~$M/N\to \alpha$, $\alpha>0$, a scaling limit emerges.  We consider a gap probability for the corresponding limiting determinantal process, namely, the probability that there are no particles in the interval~$(a,+\infty)$. We derive a Tracy--Widom-like formula in terms of the unique solution of a certain matrix Riemann--Hilbert problem of size~$2 \times 2$. The right-tail asymptotics for this solution is obtained by the Deift--Zhou non-linear steepest descent analysis.
\end{abstract}

\maketitle
\section{Introduction}
Studying gap probabilities for determinantal and Pfaffian point processes is of interest to several fields of mathematics, such as random matrix theory, combinatorics, statistical mechanics, and representation theory. Gap probabilities quantify how likely it is to observe no particles in a given interval of the real line.

In the scope of random matrix theory, the results for gap probabilities, exact and asymptotic, can be found in Deift, Its, and Krasovsky~\cite{DeiftItsKrasovsky}, in Forrester~\cite[Chapters 8, 9]{Forrester}, in Forrester and Witte~\cite{ForresterWitteI, ForresterWitteII}, in Krasovsky~\cite{Krasovsky}, and in Tracy and Widom~\cite{TracyWidomI,TracyWidomII, TracyWidomIII, TracyWidomIV}; see also references therein. Applications of gap probabilities and related quantities to combinatorics and statistical mechanics are discussed in Deift~\cite{Deift} and in Baik, Deift, and Suidan~\cite{BaikDeiftSuidan}. In the context of representation theory of the infinite symmetric and unitary groups, the gap probabilities are studied in Borodin and Deift~\cite{BorodinDeift} and in Deift, Krasovsky, and Vasilevska~\cite{DeiftKrasovskyVasilevska}.

In this paper we consider point processes formed by singular values of the products of i.i.d. complex Gaussian matrices with i.i.d. entries. The determinantal structure of such processes was first established in Akemann and Burda~\cite{AkemannBurda}, in Akemann, Kieburg, and Wei~\cite{AkemannKieburgWei}, and in Akemann, Ipsen, and Kieburg~\cite{AkemannIpsenKieburg}. Originally, the corresponding correlation kernels were given in terms of the Meijer G-functions; later on, an alternative double contour integral representation was discovered by Kuijlaars and Zhang~\cite{KuijlaarsZhang}.

Scaling limits of the kernels related to the products of random matrices produces a variety of new limiting kernels and thus a variety of new determinantal point processes. A manifestation of the universality and fundamental importance of these determinantal processes lies in the fact that they are also scaling limits of combinatorial models without a priori relation to random matrices, much like it is the case for the classical Airy, Bessel, and sine processes. For further information see Ahn~\cite{Ahn}, and Borodin, Gorin and Strahov~\cite{BorodinGorinStrahov}.

We are concerned with the scaling limit near the soft edge under the condition that both the matrix size and the number of factors in the product approach infinity in such a way that their ratio converges to a positive number. The study of this asymptotic regime was initiated by Akemann, Burda, and  Kieburg in~\cite{AkemannBurdaKieburg1, AkemannBurdaKieburg2}; the scaling limit of the corresponding kernel was obtained by Liu, Wang, and Wang~\cite{LiuWangWang}. Following the latter paper, we refer to this limiting kernel as to the~\textit{critical kernel} and to the corresponding process as to the~\textit{critical determinantal process}. We examine a special case of the gap probability for the critical process, the probability of having no particles in the interval~$(a,+\infty)$.

The general fact of the matter is that gap probabilities of determinantal processes can be represented as Fredholm determinants of trace-class operators. Moreover, if the operators involved are of special~\textit{integrable form} (see~Its, Izergin, Korepin, and Slavnov~\cite{ItsIzerginKorepinSlavnov}), then one can relate the determinants to the Riemann--Hilbert problems. For the classical determinantal processes this leads to the theory of Painlev\'{e} transcendents (e.g., see Baik, Deift, and Suidan~\cite[Section 6.5]{BaikDeiftSuidan}, and Deift \cite{Deift}). Having said that, we emphasize that the critical kernel is not of the integrable form, and thus the standard techniques cannot be applied.

Instead, we use an alternative approach based on the earlier work by Bertola and Cafasso~\cite{BertolaCafasso}, and later by Girotti~\cite{Girotti1, Girotti2, Girotti3}. This approach was subsequently used for studying determinantal point processes related to products of random matrices in Claeys, Girotti, and Stivigny~\cite{ClaeysGirottiStivigny}, and in Charlier, Lenells, and Mauersberger~\cite{CharlierLenellsMauersbergerI, CharlierLenellsMauersbergerII}. The core idea is to show that the Fredholm determinant of the non-integrable operator can be replaced by that of an integrable operator. Then, we can evaluate the gap probability in terms of the unique solution of a Riemann--Hilbert problem and, by means of the non-linear steepest descent method of Deift and Zhou~\cite{DeiftZhou}, derive the right-tail asymptotics for this solution (see Theorem~\ref{TheoremMainResult} below).

The organization of the rest of the paper is as follows. In \textbf{Section~\ref{Section2}}, we describe determinantal processes related to the products of random matrices. In particular, the critical determinantal process is defined and the formula for the critical kernel is presented, as stated in Liu, Wang, and Wang~\cite{LiuWangWang}. In \textbf{Section~\ref{Section3}}, we formulate Proposition~\ref{PropositionTraceClass} and Theorem~\ref{TheoremMainResult}. The proposition says that after conjugating by a positive function, the critical kernel becomes a kernel of a trace-class operator on~$L_2(a,+\infty)$; this, in particular, ensures that the corresponding gap probability can be written as a Fredholm determinant. Theorem~\ref{TheoremMainResult} is the main result of our paper. It gives a Tracy--Widom-like formula for the gap probability in terms of the unique solution of a certain $2\times2$-matrix Riemann--Hilbert problem and the right-tail asymptotics of this solution. In \textbf{Section~\ref{Section4}}, we discuss Theorem~\ref{TheoremMainResult} and earlier related works. In particular, we compare formulas of Theorem~\ref{TheoremMainResult} with those of Tracy and Widom~\cite{TracyWidomI} for the Airy kernel. \textbf{Sections~\ref{Section5}--\ref{Section9}} present the proofs. In \textbf{Section~\ref{Section5}}, we establish Proposition~\ref{PropositionTraceClass}. In \textbf{Section~\ref{Section6}} we show that the gap probability can be written as a Fredholm determinant of an integrable operator. In \textbf{Section~\ref{Section7}} we represent the gap probability in terms of a matrix Riemann--Hilbert problem of size~$2\times2$. In \textbf{Section~\ref{Section8}}, we carry out the asymptotic analysis of this problem. Finally, in \textbf{Section~\ref{Section9}} we proof Theorem~\ref{TheoremMainResult}.    
\\[2ex]
\textbf{Acknowledgments.}
We are grateful to Christophe Charlier for helpful discussions and for pointing out related references. We also thank the anonymous referees for carefully reviewing the manuscript and for giving us constructive and useful remarks. This work is supported by the BSF grant 2018248 ``Products of random matrices via the theory of symmetric functions''.


\section{Determinantal processes related to products of random matrices}
\label{Section2}
We begin by describing the relevant determinantal point processes. Let~$X_1, X_2, \ldots, X_M$ be i.i.d. Ginibre matrices; their entries are i.i.d. standard complex Gaussian random variables. Assume that the~$X_j$ are all of the same size~$N\times N$. The squared singular values of~$\Pi_{N,M}= X_MX_{M-1}\cdot\ldots \cdot X_1$ form a determinantal process~$\mathscr{X}_{N,M}$, as shown in Akemann, Kieburg, and Wei~\cite{AkemannKieburgWei} and in Akemann, Ipsen, and Kieburg~\cite{AkemannIpsenKieburg}. The correlation kernel of~$\mathscr{X}_{N,M}$ has the following contour integral representation (see Kuijlaars and Zhang~\cite[Proposition 5.1.]{KuijlaarsZhang}),
\begin{equation}
  \label{KZKERNELFINITE}
  K_{N,M}(x,y)=\frac{1}{(2\pi i)^2}\int\limits_{-\frac{1}{2}-i\infty}^{-\frac{1}{2}+i\infty}ds
  \int\limits_{\Sigma_N}dt\left(\frac{\Gamma(s+1)}{\Gamma(t+1)}\right)^{M+1}\frac{\Gamma(t-N+1)}{\Gamma(s-N+1)} \frac{x^ty^{-s-1}}{s-t},\quad x,y>0,
\end{equation}
where~$\Sigma_N$ is a simple closed contour encircling~$0,1,\ldots,N-1$ in the positive direction and is such that~$\re{t}>-1/2$ for~$t\in\Sigma_N$.

For convenience, we transform~$\mathscr{X}_{N,M}$ and consider the process~$\tilde{\mathscr{X}}_{N,M}$ formed by the eigenvalues of~$\log{(\Pi_{N,M}^* \Pi_{N,M})}$. It is not hard to see that this process is also determinantal and that its correlation kernel~$\tilde{K}_{N,M}(x,y)$ can be evaluated in terms of~$K_{N,M}(x,y)$. After simple manipulations with the contours, one finds that (see Liu, Wang, and Wang~\cite[Section 1]{LiuWangWang})
\begin{equation}
  \label{init_kern}
  \tilde{K}_{N,M}(x,y) =\frac{1}{(2\pi i)^2}\int \limits_{\tilde{\gamma}}ds \int \limits_{\gamma} dt  \left(\frac{\Gamma{(s + N)}}{\Gamma{(t+N)}}\right)^{M+1}\frac{\Gamma{(t)}}{\Gamma{(s)}}
  \frac{e^{x t- y s}}{s-t} , \quad x,y \in \mathbb{R},
\end{equation}
where~$\gamma$ and~$\tilde{\gamma}$ are shown in Fig.~\ref{contours_init}.
\begin{figure}[ht!]
  \centering
  \begin{tikzpicture}[scale=2.5]
    \begin{scope}[compass style/.style={color=black}, color=black, decoration={markings,mark= at position 0.5 with {\arrow{stealth}}}]
      \draw[->] (-0.75,0) -- (1,0);
      \draw (-1.8,0) -- (-1.25,0);
      \draw[->] (0,-0.6) -- (0,0.6);

      \draw[thick, postaction=decorate] (0.5,-0.6) -- (0.5,0.6);

      \draw[thick, postaction=decorate] (-1.5,-0.25) -- (0,-0.25);
      \draw[thick, postaction=decorate] (0,0.25) -- (-1.5,0.25);
      \draw[thick, postaction=decorate] (0,-0.25)  arc[start angle=-90, end angle=90,radius=0.25];
      \draw[thick, postaction=decorate] (-1.45,0.25)  arc[start angle=90, end angle=270,radius=0.25];

      \node at (-0.05,-0.11) {0};
      \node at (0.56,-0.11) {1};
      \node at (0.58,+0.5) {$\tilde{\gamma}$};
      \node at (-0.5,+0.35) {$\gamma$};
      
      \draw (0,0) node [cross] {};
      \draw (-0.5,0) node [cross] {};
      \node at (-0.52,-0.11) {-1};
      \draw (-1,0) node {\ldots};
      \draw (-1.5,0) node [cross] {};
      \node at (-1.45,-0.11) {-N+1};

    \end{scope}
  \end{tikzpicture}  
  \caption{The contours~$\gamma$ and~$\tilde{\gamma}$ for the kernel~\eqref{init_kern}.}
  \label{contours_init}
\end{figure}
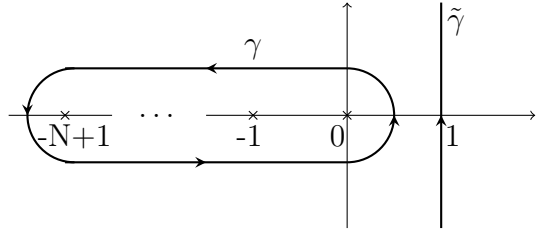
We mention that~$\tilde{\mathscr{X}}_{N,M}$ is also of relevance in the context of dynamical systems; for further information, see \cite[Section 1]{LiuWangWang}.

The starting point of our work is the following theorem. 
\begin{thm}
  Let~$M=M(N)$ be a function of~$N$ such that~$M(N)/N\to\alpha$, $\alpha>0$, as~$N \to \infty$. The following limit exists,
  \begin{equation}
    \label{th21_scaling_softedge}
    \lim\limits_{N\to\infty}\tilde{K}_{N,M(N)}(x+a_N, y+a_N)=K_{\crit}(x,y),
  \end{equation}
  where~$a_N= \left(M(N)+1\right)\left(\log N-1/(2N)\right)$, and the convergence is uniform for~$x$, $y$ in compact subsets of~$\R$. The kernel~$\tilde{K}_{N,M}(x,y)$ is defined in~\eqref{init_kern}, and the limiting kernel~$K_{\crit}(x,y)$ reads
  \begin{equation}
    \label{eq_crit_kern_init}
    K_{\crit}(x,y) =\frac{1}{(2\pi i)^2} \int \limits_{\tilde{\gamma}} ds \int \limits_{\gamma} dt \;\frac{\Gamma{(t)}}{\Gamma{(s)}} \frac{\exp{\Big(\frac{\alpha s^2}{2} - y s\Big)}}{\exp{\Big(\frac{\alpha t^2}{2}-x t\Big)}} \frac{1}{s-t}, \quad x,y\in\mathbb{R},
  \end{equation}
  where~$\gamma$ and~$\tilde{\gamma}$ are set out in Fig.~\ref{contours}.
\end{thm}
\begin{proof}
  See Liu, Wang, and Wang~\cite[Section 2]{LiuWangWang}.
\end{proof}
\begin{figure}[ht!] 
  \centering
  \begin{tikzpicture}[scale=2.5]
    \begin{scope}[compass style/.style={color=black}, color=black, decoration={markings,mark= at position 0.5 with {\arrow{stealth}}}]
      \draw[->] (-1.75,0) -- (1,0);
      \draw[->] (0,-0.6) -- (0,0.6);

      \draw[thick, postaction=decorate] (0.5,-0.6) -- (0.5,0.6);

      \draw[thick, postaction=decorate] (-1.75,-0.25) -- (0,-0.25);
      \draw[thick, postaction=decorate] (0,0.25) -- (-1.75,0.25);
      \draw[thick, postaction=decorate] (0,-0.25)  arc[start angle=-90, end angle=90,radius=0.25];

      \node at (-0.05,-0.11) {0};
      \node at (0.56,-0.11) {1};
      \node at (0.58,+0.5) {$\tilde{\gamma}$};
      \node at (-0.5,+0.35) {$\gamma$};
      
      \draw (0,0) node [cross] {};
      \draw (-0.5,0) node [cross] {};
      \node at (-0.52,-0.11) {-1};
      \draw (-1,0) node [cross] {};
      \node at (-1.02,-0.11) {-2};
      \draw (-1.5,0) node [cross] {};
      \node at (-1.52,-0.11) {-3};

    \end{scope}
  \end{tikzpicture}  
  \caption{The contours~$\gamma$ and~$\tilde{\gamma}$ for the kernel~\eqref{eq_crit_kern_init}.}
  \label{contours}
\end{figure}
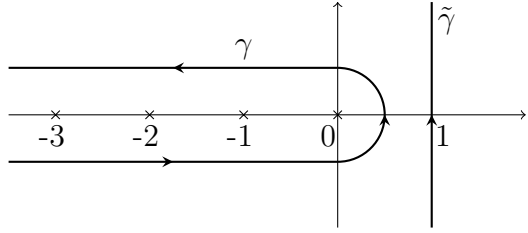
It is not hard to realize that the kernel~$K_{\crit}(x,y)$ defines a determinantal point process~$\mathscr{X}_{\crit}$ on~$\R$. We further refer to~$\mathscr{X}_{\crit}$ and~$K_{\crit}(x,y)$ as to the~\textit{critical determinantal process} and the~\textit{critical kernel}, respectively.

Our interest in~$\mathscr{X}_{\crit}$ is due to the following. First, consider the case such that the number~$M$ of factors in the product~$\Pi_{N,M}$ is fixed. The global behavior of the spectrum of~$(\Pi_{N,M}^* \Pi_{N,M})$ is well understood, and after the appropriate scaling the corresponding counting measure has a limiting density as~$N \to \infty$ (e.g., see Liu, Wang, and Zhang~\cite{LiuWangZhang}). The scaling limit at the soft edge yields the Airy determinantal point process. On the other hand, if both~$M$ and~$N$ approach infinity at an equal rate, a similar scaling limit produces another (critical) process, $\mathscr{X}_{\crit}$, which therefore can be regarded as a deformation of the Airy process. Note that the critical process~$\mathscr{X}_{\crit}$ can also be scaled back into the Airy process (see Liu, Wang, and Wang~\cite[Theorem 3.2]{LiuWangWang}).

\section{Main result}
\label{Section3}
Let~$\mathscr{X}_{\crit}$ be defined as in Section~\ref{Section2}. Set~$\mathcal{P}(a)$ to be the gap probability of having no particles of~$\mathscr{X}_{\crit}$ in the interval~$(a, +\infty)$,
\begin{equation}
  \label{gap_prob_to_find}
  \mathcal{P}(a) = \prob[\mathscr{X}_{\crit}]{\sum_{x \in \mathscr{X}_{\crit}} {1}_{(a, +\infty)}(x) =0}, \quad a \in \mathbb{R},
\end{equation}
where the probability is calculated under the measure induced by~$\mathscr{X}_{\crit}$.

The standard formula (e.g., see Anderson, Guionnet, and Zeitouni~\cite[Lemma 3.2.4]{AndersonGuionnetZeitouni}) for the gap probability is as follows,
\begin{equation}
  \label{Pset}
  \mathcal{P}(a)=1+\sum\limits_{n=1}^{\infty}\frac{(-1)^n}{n!} \int\limits_a^{+\infty}\ldots\int\limits_a^{+\infty} \det\left(K_{\crit}\left(x_j,x_k\right)\right)_{j,k=1}^n\,dx_1\cdot\ldots \cdot dx_n.
\end{equation}
Conjugating~$K_{\crit}(x,y)$ defined in~\eqref{eq_crit_kern_init} by a positive function, we have a new kernel~$K(x,y)$,
\begin{equation}
  \label{Knew}
  K(x,y)=e^{-\frac{(x-y)}{2}}K_{\crit}(x,y).
\end{equation}
Clearly, plugging~$K(x,y)$ in place of~$K_{\crit}(x,y)$ on the right-hand side of~\eqref{Pset} does not change~$\mathcal{P}(a)$.

\begin{prop}
  \label{PropositionTraceClass}
  Let~$a \in \mathbb{R}$ be fixed. The kernel~$K(x,y)$ given by~\eqref{Knew} defines a trace-class operator~$\op{K}|_{L_2(a,+\infty)}$ on~$L_2(a,+\infty)$. Moreover,
  \begin{equation}
    \label{eq:ineq_for_K}
    \left|K(x,y)\right|\leq C \myexp{-\frac{x+y}{2}}, \quad x,y > a,
  \end{equation}
  for some constant~$C>0$.
\end{prop}
A proof of this proposition is based on several estimates for the kernel~$K(x,y)$ and will be given later in Section~\ref{Section5}.

From now on, we denote all operators by the same letters as their kernels; however, in the former case we use the Roman typeface style, as opposed to the italics in the latter case.

Proposition~\ref{PropositionTraceClass} implies that~\eqref{Pset} can be written as a Fredholm determinant,
\begin{equation}
  \label{PFredholmFirst}
  \mathcal{P}(a)= \fdet{\op{I}-\op{K}|_{L_2(a,+\infty)}}, \quad a \in \mathbb{R}.
\end{equation}

In order to present our main result we state the following Riemann--Hilbert problem.
\begin{problem}[Y-RH]
  \needspace{5ex}
  \namedlabel{problem_y-rh}{Problem Y-RH}
  \leavevmode
  \begin{enumerate}[label=\textnormal{({\roman*})},ref=Y-RH-{\roman*}]
  \item  \label{yrh_cond1} $Y(z)$ is analytic in~$\mathbb{C} \setminus (\gamma \cup \tilde{\gamma})$;
  \item  \label{yrh_cond2} $Y^+(z) = Y^-(z) J_Y(x),\quad z \in \gamma \cup \tilde{\gamma}$,
    \begin{equation}
      \label{eq_jump_Y}
      J_Y(z) =
      \begin{pmatrix}
        1 &  {1}_{\tilde{\gamma}}(z)(\Gamma{(z)})^{-1} e^{\frac{\alpha z^2}{2} - a z}\\
        - {1}_{\gamma}(z) \Gamma{(z)} e^{-\frac{\alpha z^2}{2} + a z}  & 1
      \end{pmatrix};
    \end{equation}
  \item  \label{yrh_cond3} $Y(z) \to I$ as~$z \to \infty$, $z \in \mathbb{C} \setminus (\gamma \cup \tilde{\gamma})$.
  \end{enumerate}
\end{problem}
By~$Y^{\pm}(z)$ we denote the~$(\pm)$-boundary values of~$Y(z)$ on the contours~$\gamma$ and~$\tilde{\gamma}$ in Fig.~\ref{contours}.

\begin{rem}
  Note that on each of the contours, $\gamma$ and~$\tilde{\gamma}$, the jump matrix~$J_Y(z)$ is triangular.
\end{rem}
\begin{rem}
  Since~$J_Y(z)$ is analytic in a neighborhood of~$\gamma \cup \tilde{\gamma}$, the matrix function~$Y(z)$ is continuous up to the boundary~$\gamma \cup \tilde{\gamma}$, thus the limits~$Y^{\pm}(z)$ and the condition~(\ref{yrh_cond2}) can be understood pointwise (e.g., see Fokas, Its, Kapaev, and Novokshenov~\cite[Chapter 3]{FokasItsKapaevNovokshenov}). 
\end{rem}
Due to the standard argument based on Liouville's theorem the solution of~\ref{problem_y-rh} is unique (e.g., see Deift~\cite[p. 44]{Deift_book}). We will establish the existence in Proposition~\ref{prop_72}.

Now, let~$Y(z)$ be the solution of the Riemann--Hilbert problem stated above, and write
\begin{equation}
  \label{Y1}
  Y(z)=I+\frac{Y_1(a)}{z}+ \frac{Y_2(a)}{z^2} + O\left(\frac{1}{z^3}\right), 
\end{equation}
as~$z \to \infty$, $z \in \mathbb{C} \setminus (\gamma \cup \tilde{\gamma})$. We can present our main result as follows.
\begin{thm}
  \label{TheoremMainResult}
  The gap probability~\eqref{Pset} is given by
  \begin{equation}
    \label{MainExactFormula}
    \mathcal{P}(a)=\fdet{\op{I}-\op{K}|_{L_2(a,+\infty)}}=\myexp{-\int\limits_a^{+\infty}(x-a)u(x)\, dx},
  \end{equation}
  where~$u(x)$ reads
  \begin{equation}
    \label{URH}
    u(x)=-\left(Y_1(x)\right)_{1,2}\left(Y_1(x)\right)_{2,1},
  \end{equation}
  the quantities~$\left(Y_1(x)\right)_{1,2}$ and~$\left(Y_1(x)\right)_{2,1}$ are matrix elements of~$Y_1(x)$ in the asymptotic expansion~\eqref{Y1} of the solution of~\ref{problem_y-rh}.

  Moreover, the function~$u(x)$ has the following right-tail asymptotics,
  \begin{equation}
    \label{Asymptotics38}
    u(x) = \frac{\myexp{-\frac{1}{2 \alpha}\left(x^2 + \left(\log{\frac{x}{\alpha}}\right)^2\right)}}{ \Gamma{(\frac{x}{\alpha})} \sqrt{2 \pi \alpha}} \left(1+O\left(\frac{(\log{x})^2}{x}\right)\right)
  \end{equation}
  as~$x\to+\infty$.
\end{thm}
A proof of Theorem \ref{TheoremMainResult} will be given in Section~\ref{Section9}. This theorem has the following corollary.
\begin{cor}
  \label{cor_main_theorem}
  Set
  \begin{equation}
    \label{eq:u0x}
    u_0(x) = \frac{\myexp{-\frac{1}{2 \alpha}\left(x^2+ \left(\log{\frac{x}{\alpha}}\right)^2\right)}}{ \Gamma{(\frac{x}{\alpha})} \sqrt{2 \pi \alpha}}.
  \end{equation}
  In these terms one has
  \begin{equation}
    \label{eq:P_asymp}
    \mathcal{P}(a) = 1 -\int\limits_a^{+\infty}(x-a) u_0(x) \, dx + O\left( \int\limits_a^{+\infty}(x-a) u_0(x) \frac{(\log{x})^2}{x} \, dx\right), \quad a \to +\infty,
  \end{equation}
  or in its cruder form
  \begin{equation}
    \label{eq:P_asymp_crude}
    \mathcal{P}(a) = 1 + O(u_0(a)),\quad a \to +\infty.
  \end{equation}  
\end{cor}
One may think that~\eqref{eq:P_asymp} and~\eqref{eq:P_asymp_crude} can be obtained directly via the Fredholm determinant expansion~\eqref{Pset}. The first non-trivial term of this expansion is estimated in Lemma~\ref{Lemma71}, and the result turns out to be much coarser than the crude estimate in~\eqref{eq:P_asymp_crude}. Improving this result involves asymptotic analysis of multiple contour integrals, which is a difficult task in itself. On the other hand our Riemann--Hilbert approach reduces the problem to analyzing only one dimensional integrals (see Proposition~\ref{Proposition82}) and gives much more than just the bound~\eqref{eq:P_asymp}. In particular, we obtain a Tracy--Widom-like formula~\eqref{MainExactFormula} for the gap probability and the asymptotics~\eqref{Asymptotics38}, from which~\eqref{eq:P_asymp} follows directly.


\section{Comments and remarks on Theorem~\ref{TheoremMainResult} and related works}
\label{Section4}
\subsection{}
\label{SectionDiscussionDistributionFunction}
The process~$\mathscr{X}_{\mathrm{crit}}$ almost surely has an infinite number of particles in~$\mathbb{R}$, since it can be shown by a standard argument that the corresponding trace of~\eqref{eq_crit_kern_init} is infinite. Proposition~\ref{PropositionTraceClass} yields
\begin{equation}
  \mexp[\mathscr{X}_{\crit}]{\#_{(a,+\infty)}(\mathscr{X}_{\crit})}=\Tr{(\op{K}|_{L_2(a,+\infty)})} < +\infty,
\end{equation}
where~$\#_{(a,+\infty)}(\mathscr{X}_{\crit})$ is the number of particles of~$\mathscr{X}_{\crit}$ in~$(a, +\infty)$ and the expectation is taken under~$\mathscr{X}_{\crit}$. The standard argument based on Markov's inequality leads to
\begin{equation}
  \prob[\mathscr{X}_{\crit}]{\#_{(a,+\infty)}(\mathscr{X}_{\crit}) < +\infty} = 1.
\end{equation}
This allows us to define an almost surely finite random variable~$x_{\mathrm{rm}}$, the coordinate of the rightmost particle of~$\mathscr{X}_{\crit}$,
\begin{equation}
  x_{\mathrm{rm}} = \max_{x \in \mathscr{X}_{\crit}} x.
\end{equation}

The gap probability~$\mathcal{P}(a)$ given in~\eqref{gap_prob_to_find} has the following interpretation in terms of~$x_{\mathrm{rm}}$, 
\begin{equation}
  \mathcal{P}(a) = \prob[\mathscr{X}_{\crit}]{x_{\mathrm{rm}} \le a},
\end{equation}
meaning that the function~$\mathcal{P}(a)$ is the distribution function of the random variable~$x_{\mathrm{rm}}$.

Let us compare~$\mathcal{P}(a)$ with another distribution function important in random matrix theory, specifically, with the Tracy--Widom distribution function~$F_{TW}(a)$ (see Tracy and Widom~\cite{TracyWidomI}). This function is given by the formula
\begin{equation}
  F_{TW}(a)=\myexp{-\int\limits_a^{+\infty}(x-a)v^2(x)\, dx},
\end{equation}
where~$v(x)$ is the Hastings--McLeod solution of the second Painlev\'{e} equation
\begin{equation}
  v''(x)=xv(x)+2v^3(x),
\end{equation}
uniquely determined by the asymptotic condition
\begin{equation}
  v(x)=\frac{1}{2\sqrt{\pi}}x^{-\frac{1}{4}}e^{-\frac{2}{3}x^{\frac{3}{2}}}\left(1+o(1)\right),\;\; x\to+\infty.
\end{equation}

We are not able to represent our function~$u(x)$ in~\eqref{MainExactFormula} as a solution of a differential equation (see Section~\ref{subsect:asympt_of_P}); however, the formula~\eqref{URH} says that~$u(x)$ is expressible in terms of the solution of~\ref{problem_y-rh}. Likewise, the Hastings--McLeod solution~$v(x)$ of the second Painlev\'{e} equation can be understood in terms of the solution of a similar Riemann--Hilbert problem (e.g., see Its~\cite[Section 9.2]{Its}).

\subsection{}
\label{SectionDiscussionIntegrableOperators}
Let~$\Gamma$ be a contour in~$\mathbb{C}$. An operator~$\op{K}$ acting on~$L_2\left(\Gamma\right)$ is called \textit{integrable} in the sense of Its, Izergin, Korepin and Slavnov (see~\cite{ItsIzerginKorepinSlavnov}) if it has a kernel of the form
\begin{equation}\label{IntegrableKernel}
  K(z,z')=\frac{\sum \limits_{j=1}^mf_j(z)g_j(z')}{z-z'},\quad z,z'\in\Gamma,
\end{equation}
for some functions~$f_j$, $g_j$, $j=1,\ldots,m<\infty$, satisfying
\begin{equation}
  \sum_{j=1}^mf_j(z)g_j(z) = 0;
\end{equation}
the corresponding kernels are also called \textit{integrable}.

Integrable operators are well-studied (e.g., see Deift \cite{Deift}), and their key property is that the corresponding resolvents can be computed explicitly in terms of the solutions of Riemann--Hilbert problems (e.g., see Baik, Deift, and Suidan~\cite[Theorem 5.21]{BaikDeiftSuidan}). Given an integrable operator, if the dimension of the associated Riemann--Hilbert problem is sufficiently small (say, two), the non-linear steepest descent method of Deift and Zhou~\cite{DeiftZhou} can be used to carry out the asymptotic analysis of this problem, and thus to obtain the asymptotics of the corresponding Fredholm determinant. The representation in terms of the Riemann--Hilbert problem also enables one to obtain a Lax pair and then a system of differential equations for the Fredholm determinant (e.g., see Baik, Deift, and Suidan~\cite[Section 6.5]{BaikDeiftSuidan}, or Borodin and Deift~\cite{BorodinDeift}).

We note that both the kernel~$K(x,y)$ given by~\eqref{Knew} and the related kernel~$K_{\crit}(x,y)$ given by~\eqref{eq_crit_kern_init} lack integrability, and thus the usual analysis does not go through. At the same time, it is worth noticing that the following formula holds (see Liu, Wang, and Wang~\cite[Section 3.1]{LiuWangWang}), 
\begin{equation}
  \label{almost_integr}
  K_{\crit}(x,y) = \frac{\alpha f_{-1}(x)g_{-1}(y) + \sum \limits_{n=0}^\infty f_n(x) g_n(y)}{y-x},
\end{equation}
where
\begin{equation}
  \label{f_g_min1}
  f_{-1}(x) = \int \limits_{\gamma} \frac{dt}{2 \pi i} \Gamma{(t)}e^{-\frac{\alpha t^2}{2}+x t}, \quad   g_{-1}(x) = \int \limits_{\tilde{\gamma}} \frac{ds}{2 \pi i} \frac{1}{\Gamma{(s)}}e^{\frac{\alpha s^2}{2}-x s},
\end{equation}
and
\begin{equation}
  f_{k}(x) = \int \limits_{\gamma} \frac{dt}{2 \pi i} \frac{\Gamma{(t)}}{t+k} e^{-\frac{\alpha t^2}{2}+x t}, \quad   g_{k}(x) = \int \limits_{\tilde{\gamma}} \frac{ds}{2 \pi i} \frac{1}{(s+k)\Gamma{(s)}}e^{\frac{\alpha s^2}{2}-x s},\quad k=1,2,\ldots,  
\end{equation}
with the contours~$\gamma$ and~$\tilde{\gamma}$ specified in Fig.~\ref{contours}. Even though~\eqref{almost_integr} resembles an integrable kernel, the crucial difference is the number of terms in the sum. That is why the standard approach relating the determinant~$\fdet{\op{I}-\op{K}|_{L_2(a,+\infty)}}$ to a Riemann--Hilbert problem still falls short.

\subsection{}
If~$M$ is fixed, then the kernel~$K_{N,M}(x,y)$ in~\eqref{KZKERNELFINITE} has a hard edge scaling limit as~$N \to \infty$, which yields a limiting determinantal process on~$\R_{>0}$ (see~ Kuijlaars and Zhang~\cite[Theorem 5.3]{KuijlaarsZhang}). The corresponding limiting kernel is of the integrable form, and it was shown by Strahov~\cite{Strahov} that a Hamiltonian system associated with a gap probability can be derived and further related to the theory of isomonodromic deformations of Jimbo, Miwa, M\^{o}ri, and Sato~\cite{JimboMiwaMoriSato}. This leads to a formula for the gap probability in terms of a solution of a system of nonlinear differential equations (see Strahov~\cite[Proposition 3.9]{Strahov}). As it was demonstrated by Witte and Forrester~\cite{WitteForrester}, in the case~$M=2$ this system can be reduced to a single differential equation and the asymptotics of the large gap probability can be computed. Similar considerations were also used by Zhang~\cite{Zhang} and by Mangazeev and Forrester~\cite{MangazeevForrester} in related problems. We note, however, that it is highly unlikely that this method can be extended to the case of the critical kernel~\eqref{almost_integr}.

\subsection{}
\label{subsect:Krasovsky}
Krasovsky~\cite{Krasovsky} and Deift, Its, and Krasovsky~\cite{DeiftItsKrasovsky} present an approach to Fredholm determinants based on approximation of limiting kernels by finite dimensional kernels. For example, the Airy kernel related to the Tracy--Widom distribution is approximated with the Christoffel--Darboux kernel corresponding to the Hermite orthogonal polynomials.

In our situation, the kernel~$K_{N,M}(x,y)$ defined by~\eqref{init_kern} can be considered as an approximation for the limiting kernel~$K_{\crit}(x,y)$. It is known (see Kuijlaars and Zhang~\cite[Proposition 5.4]{KuijlaarsZhang}) that~$K_{N,M}(x,y)$ is of the integrable form. However, the number of terms in the enumerator of~\eqref{IntegrableKernel} will be equal to~$M+1$, and therefore the corresponding Riemann--Hilbert problem will be of size~$(M+1) \times (M+1)$. As~$M$ is growing, the size of the Riemann--Hilbert problem grows as well, and its asymptotic analysis does not seem to be feasible.

\subsection{}
\label{subsect:Girotti}
In the recent paper by Claeys, Girotti, and Stivigny~\cite{ClaeysGirottiStivigny}, an alternative approach is employed to calculating the asymptotics of a Fredholm determinant for the gap probability of the hard edge scaling limit of~$\mathscr{X}_{N,M}$ given by~\eqref{init_kern}, in the situation when~$M$ is fixed. This approach follows the earlier works by Bertola and Cafasso~\cite{BertolaCafasso} and by Girotti~\cite{Girotti1, Girotti2, Girotti3}. As a result, the large gap asymptotics at the hard edge for an arbitrary~$M$ is derived, generalizing the result of Witte and Forrester in~\cite{WitteForrester}. The result in~\cite{ClaeysGirottiStivigny} has recently been extended to include a full asymptotics expansion (with the  explicit constant term) for the gap probability by Charlier, Lenells and Mauersberger~\cite{CharlierLenellsMauersbergerI, CharlierLenellsMauersbergerII}, in an even more general setting.

The technique in~\cite{ClaeysGirottiStivigny} is based on the operator transformation which reduces the corresponding kernel to the integrable form~\eqref{IntegrableKernel}. It is also important to note that the corresponding Riemann--Hilbert problem turns out to be of size~$2\times2$, which substantially simplifies the subsequent steepest decent analysis. Our proof of Theorem~\ref{TheoremMainResult} uses the ideas similar to those of~\cite{BertolaCafasso, ClaeysGirottiStivigny, Girotti1,Girotti2, Girotti3}.

\subsection{}
An essential ingredient of the proof of our main result, Theorem~\ref{TheoremMainResult}, is the steepest descent analysis of~\ref{problem_y-rh}. There are several aspects of this analysis that make it different from the usual scenario considered in the literature.

Up to a simple normalization (see Section~\ref{Section8}) the jump matrix~\eqref{eq_jump_Y} of~\ref{problem_y-rh} is already close to the identity matrix and thus has a form suitable for applying the small norm theory (e.g., see Deift~\cite[Chapter 7]{Deift_book}). There is no need to construct the~$g$-function and the corresponding local and global parametrices. In that respect, our problem resembles the corresponding problem in the case of the Airy kernel (e.g., see Its~\cite{Its}). At the same time, one still needs to choose the normalization to guarantee that the jump matrix is as close to the identity matrix as possible. This is done in the steepest-descent fashion and thus justifies the name of the whole procedure.

However, even when the normalization has been chosen, accurate bounds on the gap probability $\mathcal{P}(a)$ are not guaranteed. For instance, the standard approach using Proposition~\ref{prop_72}, formula~\eqref{Relation87}, and the asymptotics 
\begin{equation}
  (U_1(a))_{1,1} = O\left(e^{-\frac{a^2}{2 \alpha}}\right), \quad a \to +\infty,
\end{equation}
from Proposition~\ref{th_RH_U_assympt} only yields the bound
\begin{equation}
\mathcal{P}(a) = 1 + O(e^{- \frac{a^2}{2 \alpha}}), \quad a \to +\infty,
\end{equation}
which is not even as strong as the crude bound~\eqref{eq:P_asymp_crude}. This is why Proposition~\ref{prop_73} is an important, non-standard, part of our Riemann--Hilbert analysis, which not only allows one to extract accurate asymptotics, but also yields a Tracy--Widom-like formula.

\subsection{}
\label{subsect:asympt_of_P}
An interesting problem is to find a system of differential equations for the function~$u(x)$ in Theorem~\ref{TheoremMainResult}, much like it is the case for the Airy kernel and the Painlev\'{e} II equation. However, this is an extremely hard problem and such equations are not available. Indeed, unlike the Airy case, the off-diagonal entries in the jump matrix~\eqref{eq_jump_Y} of~\ref{problem_y-rh} have an infinite number of zeros and poles in the complex plane. In particular, the standard approach (e.g., see Its~\cite{Its}) to finding the Lax pair is hard to implement fully because instead of rational, a meromorphic function needs to be determined.

Another interesting problem is to study the left-tail asymptotics for the function~$u(x)$ and thus that for the gap probability~$\mathcal{P}(a)$. Since our results ensures the reducibility to the integrable operator (Theorem~\ref{TheoremGapProbabilityIntegrableOperator}) and further to the Riemann--Hilbert problem (Proposition~\ref{prop_72}) for all~$a \in \mathbb{R}$, the consequent Riemann--Hilbert analysis is expected to be standard (the construction of the~$g$-function and the local and global parametrices) but of greater technical difficulty. This is a subject of our future work.


\section{Proof of Proposition~\ref{PropositionTraceClass}}
\label{Section5}
Let us write out the kernel~\eqref{Knew} using~\eqref{eq_crit_kern_init}. We have
\begin{equation}
  \label{K41}
  K(x,y) =\frac{1}{(2\pi i)^2} \int \limits_{\tilde{\gamma}} ds \int \limits_{\gamma} dt \;\frac{\Gamma{(t)}}{\Gamma{(s)}}\; \frac{\exp{\Big(\frac{\alpha s^2}{2} - y \big(s-\frac{1}{2}\big)\Big)}}{\exp{\Big(\frac{\alpha t^2}{2} - x \big(t-\frac{1}{2}\big)\Big)}}\frac{1}{s-t},\quad x,y > a,
\end{equation}
where the integration contours~$\gamma$ and~$\tilde{\gamma}$ are shown in Fig.~\ref{contours}.

Our first task is to prove that the double contour integral in~\eqref{K41} converges absolutely. Set $g(s,t;x,y)$ to be the integrand on the right-hand side of~\eqref{K41},
\begin{equation}
  g(s,t;x,y)=\frac{\Gamma{(t)}}{\Gamma{(s)}}\; \frac{\exp{\Big(\frac{\alpha s^2}{2} - y \big(s-\frac{1}{2}\big)\Big)}}{\exp{\Big(\frac{\alpha t^2}{2} - x \big(t-\frac{1}{2}\big)\Big)}} \frac{1}{s-t}.
\end{equation}
We write
\begin{equation}
  \label{eq:g_via_g1g2}
  g(s,t;x,y)=\frac{g_1(t)g_2(s)}{s-t},
\end{equation}
where
\begin{equation}
  \begin{aligned}
    g_1(t)=\Gamma(t)\myexp{-\frac{\alpha t^2}{2}+x \Big(t-\frac{1}{2}\Big)},\quad
    g_2(s)=\left(\Gamma(s)\right)^{-1}\myexp{\frac{\alpha s^2}{2}-y\Big(s-\frac{1}{2}\Big)}.
  \end{aligned}
\end{equation}

Recall Stirling's formula for the gamma function,
\begin{equation}\label{StirlingFormula}
  \Gamma(z)=\sqrt{2 \pi} e^{-z}z^{z-\frac{1}{2}}\left(1+O\left(\frac{1}{z}\right)\right)
\end{equation}
as~$z\to\infty$, $|\arg{z}| < \pi$.

For sufficiently large~$|t|$, the contour~$\gamma$ can be parameterized by~$t(r)=-r\pm i\delta$, where~$r>0$ and~$\delta>0$ is fixed. Also, for sufficiently large~$|s|$, the contour~$\tilde{\gamma}$ can be parameterized by~$s(\tau)=1 + i \tau$, where~$\tau\in\R$. Stirling's formula~\eqref{StirlingFormula} then implies
\begin{equation}
  g_1(t)=O\left(e^{-C|t|^2}\right)
\end{equation}
for some~$C>0$, as~$t\to\infty$, $t \in \gamma$, uniformly in~$x>a$; and
\begin{equation}
  g_2(s)=O\left(e^{-C|s|^2}\right) 
\end{equation}
for some~$C>0$, as~$s\to\infty$, $s \in \tilde{\gamma}$, uniformly in~$y>a$.

Since~$\gamma \cap \tilde{\gamma} = \varnothing$, the difference~$(s-t)$ in~\eqref{eq:g_via_g1g2} is bounded from below, and we can conclude that
\begin{equation}
  g(s,t;x,y)=O\left(e^{-C\left(|t|^2+|s|^2\right)}\right)
\end{equation}
as~$t,s \to \infty$, $t \in \gamma$, $s \in \tilde{\gamma}$, uniformly in~$x,y>a$. Thus the double contour integral in~\eqref{K41} converges absolutely.

Next, observe that the kernel~\eqref{K41} can be written as
\begin{equation}
  \label{KGG}
  K(x,y)=\int\limits_0^{+\infty}G(x,q)\widetilde{G}(q,y)\, dq,
\end{equation}
with~$G(x,q)$ and~$\widetilde{G}(q,y)$ given by
\begin{equation}
  \label{G45}
  G(x,q)=\frac{1}{2\pi i}\int\limits_{\gamma}dt\, \Gamma(t)e^{-\frac{\alpha t^2}{2}+(x+q)\left(t-\frac{1}{2}\right)}
\end{equation}
and
\begin{equation}
  \label{G46}
  \widetilde{G}(q,y)=\frac{1}{2\pi i}\int\limits_{\tilde{\gamma}}ds\, \left(\Gamma(s)\right)^{-1}
  e^{\frac{\alpha s^2}{2}-(y+q)\left(s-\frac{1}{2}\right)}.
\end{equation}
Indeed, since for fixed~$x, y> a$ we have
\begin{equation}
  \frac{\Gamma(t)}{\Gamma(s)} \; \frac{\exp{\Big(\frac{\alpha s^2}{2} - y \big(s-\frac{1}{2}\big)+q t\Big)}}{\exp{\Big(\frac{\alpha t^2}{2} - x \big(t-\frac{1}{2}\big)+q s\Big)}}=O\left(e^{-C\left(|s|^2+|t|^2\right)-\frac{q}{2}}\right)
\end{equation}
as~$s,t,q \to \infty$, $s \in \tilde{\gamma}$, $t \in \gamma$, $q>0$, Fubini's theorem yields~\eqref{KGG} immediately.

The next step is to show that the kernels~$G(x,q)$ and~$\widetilde{G}(q,y)$ define Hilbert--Schmidt operators,~$\op{G}$ and~$\widetilde{\op{G}}$, from~$L_2(0,+\infty)$ to~$L_2(a,+\infty)$ and from~$L_2(a,+\infty)$ to~$L_2(0,+\infty)$, respectively. Note that Stirling's formula~\eqref{StirlingFormula} yields
\begin{equation}
  \label{Ineq41}
  \left(\Gamma(s)\right)^{-1}e^{\frac{\alpha s^2}{2}-(y+q)\left(s-\frac{1}{2}\right)} = O\left(e^{-C|s|^2-\frac{y+q}{2}}\right)
\end{equation}
as~$s \to \infty$, $s \in \tilde{\gamma}$, uniformly in~$y>a$,~$q>0$. Thus,
\begin{equation}\label{GTILDEESTIMATE}
  \widetilde{G}(q,y) = O\left(e^{-\frac{y+q}{2}}\right)
\end{equation}
as~$q,y \to +\infty$. This means~$\widetilde{G}(q,y)$ is Hilbert--Schmidt.

To prove that~$G(x,q)$ is Hilbert--Schmidt, we split the contour~$\gamma$ into two parts as shown in Fig.~\ref{contours_break}
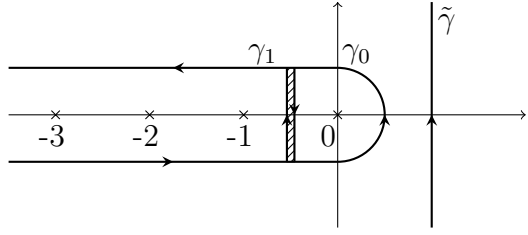
\begin{figure}[ht!]
  \centering
  \begin{tikzpicture}[scale=2.5]
    \begin{scope}[compass style/.style={color=black}, color=black, decoration={markings,mark= at position 0.5 with {\arrow{stealth}}}]
      \draw[->] (-1.75,0) -- (1,0);
      \draw[->] (0,-0.6) -- (0,0.6);

      \draw[thick, postaction=decorate] (0.5,-0.6) -- (0.5,0.6);

      \draw[thick, postaction=decorate] (-1.75,-0.25) -- (0,-0.25);
      \draw[thick, postaction=decorate] (0,0.25) -- (-1.75,0.25);
      \draw[thick, postaction=decorate] (0,-0.25)  arc[start angle=-90, end angle=90,radius=0.25];

      \draw[thick, postaction=decorate] (-0.27,-0.25) -- (-0.27,0.25);
      \draw[thick, postaction=decorate]  (-0.23,0.25) -- (-0.23,-0.25);
      \draw[fill, pattern=north east lines] (-0.23,0.25) -- (-0.23,-0.25) -- (-0.27,-0.25) -- (-0.27,0.25);

      \node at (-0.05,-0.11) {0};
      \node at (0.58,+0.5) {$\tilde{\gamma}$};
      \node at (-0.4,+0.32) {$\gamma_1$};
      \node at (0.1,+0.32) {$\gamma_0$};
      
      \draw (0,0) node [cross] {};
      \draw (-0.5,0) node [cross] {};
      \node at (-0.52,-0.11) {-1};
      \draw (-1,0) node [cross] {};
      \node at (-1.02,-0.11) {-2};
      \draw (-1.5,0) node [cross] {};
      \node at (-1.52,-0.11) {-3};

    \end{scope}
  \end{tikzpicture}  
  \caption{The contours~$\gamma = \gamma_0 \cup \gamma_1$ and~$\tilde{\gamma}$. The contour~$\gamma_0$ encircles zero, and the contour~$\gamma_1$ encircles negative integers. The contours~$\gamma_0$ and~$\gamma_1$ meet at~$\re{t} = -1/2$.}
  \label{contours_break}
\end{figure}
and write
\begin{equation}
  \label{Gsum}
  G(x,q)=\frac{1}{2\pi i}\int\limits_{\gamma_0}dt\, \Gamma(t)e^{-\frac{\alpha t^2}{2}+(x+q)\left(t-\frac{1}{2}\right)}+
  \frac{1}{2\pi i}\int\limits_{\gamma_1}dt\, \Gamma(t)e^{-\frac{\alpha t^2}{2}+(x+q)\left(t-\frac{1}{2}\right)}.
\end{equation}
The contour integral along~$\gamma_0$ can be computed explicitly by the residue theorem,
\begin{equation}
  \label{Eq49}
  \frac{1}{2\pi i}\int\limits_{\gamma_0}dt\, \Gamma(t)e^{-\frac{\alpha t^2}{2}+(x+q)\left(t-\frac{1}{2}\right)}=e^{-\frac{x+q}{2}}.
\end{equation}
To estimate the other integral in~\eqref{Gsum}, we use~\eqref{StirlingFormula} once again and write
\begin{equation}
  \Gamma(t)e^{-\frac{\alpha t^2}{2}+(x+q)\left(t-\frac{1}{2}\right)} = O\left(e^{-C|t|^2-\frac{x+q}{2}}\right)
\end{equation}
as~$t \to \infty$, $t \in \gamma_1$, uniformly in~$x>a$, $q>0$. This implies
\begin{equation}
  \label{Ineq410}
  \frac{1}{2\pi i}\int\limits_{\gamma_1}dt\, \Gamma(t)e^{-\frac{\alpha t^2}{2}+(x+q)\left(t-\frac{1}{2}\right)} =  O\left(e^{-\frac{x+q}{2}} \right),
\end{equation} 
as~$x,q \to +\infty$. Gathering up~\eqref{Eq49} and~\eqref{Ineq410}, we arrive at
\begin{equation}
  \label{E411}
  G(x,q) = O\left(e^{-\frac{x+q}{2}}\right)
\end{equation}
as~$x,q \to +\infty$. This proves the claim.

Clearly, the identity~\eqref{KGG} translates into
\begin{equation}
  \op{K} = \op{G} \tilde{\op{G}}.
\end{equation}

Since~$\op{G}$ and~$\op{\tilde{G}}$ are Hilbert--Schmidt, $\op{K}$ is of trace class (e.g., see Baik, Deift, and Suidan~\cite[Property (A.11)]{BaikDeiftSuidan}). Moreover, ~\eqref{Eq49} and~\eqref{E411}, together with~\eqref{KGG}, show that~\eqref{eq:ineq_for_K} holds. This concludes the proof of the proposition.
\qed


\section{Reduction to an integrable kernel}
\label{Section6}
Recall that due to~\eqref{PFredholmFirst}, the gap probability~$\mathcal{P}(a)$ is a Fredholm determinant of the operator~$\op{K}|_{L_2(a,+\infty)}$. It is possible to transform this operator into an integrable operator without changing the determinant.
\begin{thm}
  \label{TheoremGapProbabilityIntegrableOperator}
  The gap probability~$\mathcal{P}(a)$ can be written as
  \begin{equation}
    \label{eq:th_P_via_Fr}
    \mathcal{P}(a)=\fdet{\op{I}-\op{Q}_a},\quad a \in \mathbb{R},
  \end{equation}
  where the operator~$\op{Q}_a$ on~$L_2\left(\gamma\cup\tilde{\gamma}\right)$ is defined by the kernel~$Q_a(x,y)$,
  \begin{equation}
    \label{KernelQa}
    Q_a(x,y)=\frac{{1}_{\tilde{\gamma}}(x){1}_{\gamma}(y) \Gamma(y) e^{\frac{\alpha( x^2 - y^2)}{4}-a(x-y)} -{1}_{\gamma}(x){1}_{\tilde{\gamma}}(y) \left(\Gamma(y)\right)^{-1} e^{-\frac{\alpha (x^2-y^2)}{4}}}{2\pi i(x-y)}, \quad x,y \in \gamma\cup\tilde{\gamma}. 
  \end{equation}
\end{thm}
\begin{rem}
  \label{Q_a_analytic}
  It can be seen from the proof of this theorem that the right-hand side of~\eqref{eq:th_P_via_Fr} is well-defined for all~$a \in \mathbb{C}$.
\end{rem}
Before we prove this theorem, we need to establish an auxiliary lemma.
\begin{lem}
  \label{LemmaOperators}
  Fix~$a \in \mathbb{R}$, and let~$\op{K}|_{L_2(a, + \infty)}$ be an integral operator whose kernel~$K(x,y)$ is given by the formula~\eqref{K41}. Then, 
  \begin{equation}
    \fdet{\op{I}-\op{K}|_{L_2(a, + \infty)}} = \fdet{\op{I}-\op{H}_a},
  \end{equation}
  where~$\op{H}_a$ is the trace-class operator on~$L_2(\tilde{\gamma})$ given by the kernel
  \begin{equation}
    \label{eq_kern_h}
    H_a(z,s) = - \frac{1}{4 \pi^2} \int \limits_{\gamma}\frac{e^{a(t-z)}}{(s-t)(z-t)}\frac{\Gamma{(t)}}{\Gamma{(s)}} \myexp{{\frac{\alpha}{4}(z^2 + s^2 -2t^2)}}\, dt.
  \end{equation}
\end{lem}
\begin{proof}
  Let us first check that~$\op{H}_a$ is indeed of trace class. Set
  \begin{equation}
    \label{63a}
    A_a(z,t) = \frac{\Gamma{(t)}}{2\pi i} \frac{e^{-a(z-t)}}{z-t} \myexp{\frac{\alpha}{4}(z^2-t^2)},
  \end{equation}
  where~$z\in\tilde{\gamma}$, $t\in\gamma$, and set
  \begin{equation}
    \label{63b}
    B(t,s) = \frac{1}{2\pi i} \frac{1}{(s-t)\Gamma{(s)}} \myexp{\frac{\alpha}{4}(s^2-t^2)}.
  \end{equation}
  where~$t\in\gamma$, $s\in\tilde{\gamma}$. It is not hard to see that
  \begin{equation}
    \int\limits_{\tilde{\gamma}}\int\limits_{\gamma}\left|A_a(z,t)\right|^2|dz||dt|<\infty,\quad \int\limits_{\gamma}\int\limits_{\tilde{\gamma}}|B(t,s)|^2|dt||ds|<\infty.
  \end{equation}
  Therefore~\eqref{63a} and~\eqref{63b} are kernels of the Hilbert--Schmidt operators~$\op{A}_a\!: L_2(\gamma)\to L_2(\tilde{\gamma})$ and $\op{B}\!: L_2\left(\tilde{\gamma}\right)\to L_2(\gamma)$. Also, clearly
  \begin{equation}
    \label{eq:H_comp_AB}
    \op{H}_a=\op{A}_a\op{B},
  \end{equation}
  which implies that~$\op{H}_a$ is of trace class.

  Next, observe that the following identity holds
  \begin{equation}
    \label{Identity61}
    {1}_{(0,+\infty)}(x)e^{xt} =\frac{1}{2\pi i}{\lim\limits_{R\to+\infty}} \int\limits_{b-iR}^{b+iR} \frac{e^{xz} \, dz}{z-t}
    =\frac{1}{2\pi i}{\vp}\int\limits_{b-i\infty}^{b+i\infty}\frac{e^{xz} \, dz}{z-t},
  \end{equation}
  where~$t\in\gamma$, $b>\re{t}$, and~$x \ne 0$. The standard way of proving this formula is to consider two cases, $x>0$ and~$x<0$, to deform the contours into half-circles in the left and right half planes, respectively, and finally to apply the residue theorem together with Jordan's lemma.

  Using~\eqref{Identity61} and~\eqref{K41}, we obtain
  \begin{equation}
    \label{K62}
    \begin{aligned}
      &1_{(a,+\infty)}(x) K(x,y)\\
      &=\frac{1}{(2\pi i)^2} \int \limits_{\tilde{\gamma}} ds \int\limits_{\gamma} dt \, \frac{\Gamma{(t)}}{\Gamma{(s)}} \frac{\exp{\Big(\frac{\alpha s^2}{2}-y\big(s-\frac{1}{2}\big)\Big)}}{\exp{\Big(\frac{\alpha t^2}{2}+\frac{x}{2}\Big)}} \frac{e^{a t}}{(s-t)}
      \left(\frac{1}{2\pi i}{\vp}\int\limits_{\tilde{\gamma}}\frac{e^{(x-a)z}\, dz}{z-t}\right).
    \end{aligned}
  \end{equation}
  Define
  \begin{equation}
    g(s,t,z)=\frac{\Gamma{(t)}}{\Gamma{(s)}} \frac{\exp{\Big(\frac{\alpha s^2}{2}-y\big(s-\frac{1}{2}\big)\Big)}}{\exp{\Big(\frac{\alpha t^2}{2}+\frac{x}{2}\Big)}} \frac{e^{a(t-z)}}{(s-t)(z-t)} e^{xz}.
  \end{equation}
  The equation~\eqref{K62} takes the form
  \begin{equation}
    1_{(a,+\infty)}(x) K(x,y)=\frac{1}{(2\pi i)^3}  \int \limits_{\tilde{\gamma}} ds \int \limits_{\gamma} dt \left({\vp}\int\limits_{\tilde{\gamma}}g(s,t,z)\, dz\right).
  \end{equation}

  Let us prove that
  \begin{equation}\label{64}
    \int \limits_{\tilde{\gamma}} ds \int \limits_{\gamma} dt
    \left({\vp}\int\limits_{\gamma} g(s,t,z)\, dz\right)
    =\vp\int\limits_{\tilde{\gamma}}dz \int \limits_{\tilde{\gamma}} ds
    \int\limits_{\gamma}g(s,t,z)\, dt.
  \end{equation}
  First, suppose that~$x<a$.  Then the identity~\eqref{Identity61} implies that the left-hand side of~\eqref{64} is zero. To see that the right-hand side of~\eqref{64} is also zero, we complete the $z$-contour~$\tilde{\gamma}$ by a half-circle with~$\re{z} \ge 1$ and apply Jordan's lemma, taking into account that
  \begin{equation}
    z\mapsto\int\limits_{\tilde{\gamma}}ds\int\limits_{\gamma}g(s,t,z)\, dt 
  \end{equation}
  is analytic for~$\re{z} > 1$ and continuous up to the boundary~$\re{z}=1$. Now, suppose that~$x>a$. Deform the contour~$\tilde{\gamma}$ in the $z$-plane into two rays coming from/going to the left half plane, not intersecting~$\gamma$, and meeting at~$z=1$. The direction on the rays is chosen coherently with that on~$\tilde{\gamma}$. This deformation renders the integral on the right-hand side of~\eqref{64} absolutely convergent, and thus Fubini's theorem allows us to change the order of the integrals. Finally, we deform the rays back into their original form and conclude that the identity~\eqref{64} also holds for~$x>a$.

  We arrive at
  \begin{equation}
    \label{K65}
    {1}_{(a,+\infty)}(x) K(x,y)=\frac{1}{(2\pi i)^3} 
    \vp\int\limits_{\tilde{\gamma}} dz \int \limits_{\tilde{\gamma}}ds \int \limits_{\gamma} dt\;\frac{\Gamma{(t)}}{\Gamma{(s)}} \frac{\exp{\Big(\frac{\alpha s^2}{2}-y\big(s-\frac{1}{2}\big)\Big)}}{\exp{\Big(\frac{\alpha t^2}{2}+\frac{x}{2}\Big)}} \frac{e^{a(t-z) + xz}}{(s-t)(z-t)}.
  \end{equation}
  This formula enables us to write~$\op{K}|_{L_2(a,+\infty)}$ as the following composition
  \begin{equation}
    \label{OperatorEquation}
    \op{K}|_{L_2(a,+\infty)}=\op{R}^{-1}\mathcal{F}^{-1}\op{U}^{-1}\op{M}^{-1}\op{H}_a\op{M}\op{U}\mathcal{F}\op{R};
  \end{equation}
  The operator~$\op{R}\!: L_2(a,+\infty) \to L_2(-\infty,+\infty)$ is defined by
  \begin{equation}
    \left(\op{R}[f]\right)(x)=e^{-\frac{x}{2}}\widetilde{f}(x),  
  \end{equation}
  where~$\widetilde{f}$ is the extension of~$f$ from~$(a,+\infty)$ to~$(-\infty,+\infty)$ by zero; the operator~$\mathcal{F}$ on~$L_2(-\infty,+\infty)$ is the Fourier transform
  \begin{equation}
    \left(\mathcal{F}[g]\right)(x)=\lim\limits_{R\to +\infty}\int_{-R}^Rg(y)e^{-ixy}dy;
  \end{equation}
  the operator~$\op{U}\!:L_2(-\infty,+\infty) \to L_2\left(\tilde{\gamma}\right)$ is defined by
  \begin{equation}
    \left(\op{U}[f]\right)(s)=g(s), \quad s\in \tilde{\gamma},
  \end{equation}
  where~$g(1+iy)=f(y)$; and finally, the operator~$\op{M}$ on~$L_2(\tilde{\gamma})$ is defined by
  \begin{equation}
    \label{eq:opM_def}
    \left(\op{M}[f]\right)(s)=e^{\frac{\alpha s^2}{4}}f(s).
  \end{equation}

  The operators~$\op{U}$ and~$\mathcal{F}$ are clearly unitary and their inverses~$\op{U}^{-1}$ and~$\mathcal{F}^{-1}$ are well-defined. The operators~$\op{R}$ and~$\op{M}$ are bounded, but their (right) inverses~$\op{R}^{-1}$ and~$\op{M}^{-1}$ are not. In particular, the domains of~$\op{R}^{-1}$ and~$\op{M}^{-1}$ are strict subsets of~$L_2(-\infty, +\infty)$ and~$L_2(\tilde{\gamma})$, respectively.

  We are going to need the following well-known property of trace-class operators. Let~$\op{A}$ be a trace-class operator acting on a complex separable Hilbert space~$\mathcal{H}$. Let~$\left\{\op{T}_n\right\}$ and~$\left\{\op{S}_n\right\}$ be sequences of bounded linear operators on~$\mathcal{H}$ such that~$\op{T}_n\to T$ and~$\op{S}_n\to S$ pointwise (i.e., in strong operator topology). Then, $\op{T}_n\op{A}\op{S}^{*}_n\to \op{T}\op{A}\op{S}$ in the trace-class norm (e.g., see Gohberg, Gohberg, and Krupnik~\cite[Theorem 11.3]{Gohberg}). This can be generalized to the case of bounded operators between different Hilbert spaces.
  
  The continuity of a Fredholm determinant in the trace-class norm and the property described above yield the identity
  \begin{equation}
    \label{eq:621}
    \fdet{\op{I}-\op{K}|_{L_2(a,+\infty)}} = \lim_{N \to \infty} 
    \fdet{\op{I} - \chi_{(a,N)}\op{R}^{-1}\mathcal{F}^{-1}\op{U}^{-1}\op{M}^{-1}\op{H}_a \op{M}\op{U}\mathcal{F}\op{R}},
  \end{equation}
  where~$\chi_{(a,N)}$ is the multiplication by the indicator~$1_{(a,N)}$ of the interval~$(a,N)$. Now, the operator~$\chi_{(a,N)}\op{R}^{-1}$ becomes bounded.  Since the left-hand side of~\eqref{OperatorEquation} is well-defined, the range of~$\mathcal{F}^{-1}\op{U}^{-1}\op{M}^{-1}\op{H}_a\op{M}\op{U}\mathcal{F}\op{R}$ must be in the domain of~$\op{R}^{-1}$ and thus in that of~$\chi_{(a,N)}\op{R}^{-1}$. Consequently, we can naturally extend the domain of~$\chi_{(a,N)} \op{R}^{-1}$ to~$L_2(-\infty,\infty)$ without changing the determinant. We keep using the same notation for the extended operator.
  
  Note that it follows from~\eqref{OperatorEquation} that
  \begin{equation}
    \chi_{(a,N)}\op{R}^{-1}\mathcal{F}^{-1}\op{U}^{-1}\op{M}^{-1}\op{H}_a \op{M}\op{U}\mathcal{F}\op{R} = \chi_{(a,N)} \op{K}|_{L_2(a,+\infty)},
  \end{equation}
  thus the operator on the left-hand side of this formula is of trace class, the right-hand side being a composition of a bounded and a trace-class operator. Also, by writing out~\eqref{eq:H_comp_AB} via~\eqref{63a}--\eqref{63b} and by using~\eqref{eq:opM_def}, one can easily prove that
  \begin{equation}
    \op{M}^{-1} \op{H}_a \op{M}
  \end{equation}
  is of trace class, and consequently so is~$\mathcal{F}^{-1}\op{U}^{-1}\op{M}^{-1}\op{H}_a \op{M}\op{U}\mathcal{F}$. This together with a general fact that if~$\op{A}$ and~$\op{B}$ are bounded linear operators such that both~$\op{A}\op{B}$ and~$\op{B}\op{A}$ are of trace class, then~$\fdet{\op{I}+\op{A}\op{B}}=\fdet{\op{I}+\op{B}\op{A}}$, shows that
  \begin{equation}
    \fdet{\op{I}-\op{K}|_{L_2(a,+\infty)}} = \lim_{N \to \infty}
    \fdet{\op{I} - \op{R}\chi_{(a,N)}\op{R}^{-1}\mathcal{F}^{-1}\op{U}^{-1}\op{M}^{-1}\op{H}_a \op{M}\op{U}\mathcal{F}}.
  \end{equation}
  Clearly,
  \begin{equation}
    R\chi_{(a,N)}R^{-1} =\chi_{(a, N)},
  \end{equation}
  and we have
  \begin{equation}
    \fdet{\op{I}-\op{K}|_{L_2(a,+\infty)}} = \lim_{N \to \infty}
    \fdet{\op{I} - \chi_{(a,N)}\mathcal{F}^{-1}\op{U}^{-1}\op{M}^{-1}\op{H}_a \op{M}\op{U}\mathcal{F}}.
  \end{equation}
  Using the continuity of Fredholm determinants once again, we find that
  \begin{equation}
    \fdet{\op{I}-\op{K}|_{L_2(a,+\infty)}} = \fdet{\op{I} - \mathcal{F}^{-1}\op{U}^{-1}\op{M}^{-1}\op{H}_a \op{M}\op{U}\mathcal{F}}.
  \end{equation}
  Now, recall that~$\mathcal{F}$ and~$\op{U}$ are unitary and thus do not change the spectra of operators. This means that the determinants are preserved under unitary conjugations, and we have
  \begin{equation}
    \fdet{\op{I}-\op{K}|_{L_2(a,+\infty)}} = \fdet{\op{I} -\op{M}^{-1}\op{H}_a \op{M}}.
  \end{equation}

  Finally, the argument used above to eliminate~$\op{R}$ carries over, mutatis mutandis, to the case of~$\op{M}$. We have
  \begin{equation}
    \begin{aligned}
      \fdet{\op{I}-\op{K}|_{L_2(a,+\infty)}}  &=\lim_{N \to \infty}\fdet{\op{I} -\chi_{(1-iN,1+iN)}\op{M}^{-1}\op{H}_a \op{M}}\\
      &=\lim_{N \to \infty}\fdet{\op{I} -\op{M}\chi_{(1-iN,1+iN)}\op{M}^{-1}\op{H}_a} \\
      &=\lim_{N \to \infty}\fdet{\op{I} -\chi_{(1-iN,1+iN)}\op{H}_a}=\fdet{\op{I}-\op{H}_a}.
    \end{aligned}
  \end{equation}
  This concludes the proof of the lemma.
\end{proof}

\begin{proof}[Proof of Theorem \ref{TheoremGapProbabilityIntegrableOperator}]
  In order to prove the theorem,  we use the formula~\eqref{eq:H_comp_AB} with~$\op{A}_a$ and~$\op{B}$ defined by the kernels~\eqref{63a} and~\eqref{63b}, respectively.

  Let us first show that both~$\op{A}_a$ and~$\op{B}$ are of trace class. We write each of them as a composition of Hilbert--Schmidt operators.

  Introduce operators~$\op{A}^{(1)}_a\!: L_2(0,+\infty) \to L_2(\tilde{\gamma})$ and~$\op{A}^{(2)}_a\!: L_2(\tilde{\gamma}) \to L_2(0,+\infty)$ by their kernels,
  \begin{equation}
    A^{(1)}_a(z,w) = \frac{1}{2 \pi i} \myexp{-w\left(z-\frac{3}{4}\right) -a z +  \frac{\alpha z^2}{4}},
  \end{equation}  
  and
  \begin{equation}
    A^{(2)}_a(w,t) = \myexp{w\left(t-\frac{3}{4}\right) +a t
      -\frac{\alpha t^2}{4}} \Gamma(t),
  \end{equation}
  where~$w>0$, $z\in\tilde{\gamma}$, and~$t\in\gamma$.  It is not hard to check that
  \begin{equation}
    A^{(1)}_a(z,w) = O(e^{-C|z|^2 -\frac{w}{4}}), \quad A^{(2)}_a(w,t) = O(e^{-C|t|^2 -\frac{w}{4}})
  \end{equation}
  for some~$C>0$, as~$z,t, w \to \infty$, $z \in \tilde{\gamma}$, $t\in\gamma$, $w > 0$. And we conclude that both~$\op{A}^{(1)}$ and~$\op{A}^{(2)}$ are Hilbert--Schmidt.

  By a direct computation, we see that
  \begin{equation}
    \op{A}_a=\op{A}^{(1)}_a\op{A}^{(2)}_a, 
  \end{equation}
  which justifies that~$\op{A}_a$ is a trace-class operator.

  In a similar way, we introduce operators~$\op{B}^{(1)}\!: L_2\left(0,+\infty\right) \to L_2(\gamma)$ and~$\op{B}^{(2)}\!: L_2(\tilde{\gamma}) \to L_2\left(0,+\infty\right)$ by their kernels
  \begin{equation}
    B^{(1)}(t,w) = \frac{1}{2 \pi i} \myexp{w\left(t-\frac{3}{4}\right) 
      - \frac{\alpha t^2}{4}},
  \end{equation}  
  and
  \begin{equation}
    B^{(2)}(w,s) = \myexp{-w\left(s-\frac{3}{4}\right) +\frac{\alpha s^2}{4}} \left(\Gamma(s)\right)^{-1}.
  \end{equation}
  We have
  \begin{equation}
    B^{(1)}(t,w)=O\left(e^{-C|t|^2-\frac{w}{4}}\right), \quad B^{(2)}(w,s)=O\left(e^{-C|s|^2-\frac{w}{4}}\right)
  \end{equation}
  as~$t, s, w \to \infty$, $t\in\gamma$, $s \in \tilde{\gamma}$, $w>0$. Again, a direct computation using kernels shows that
  \begin{equation}
    \op{B}=\op{B}^{(1)}\op{B}^{(2)},
  \end{equation}
  and we conclude that~$\op{B}$ is of trace class. 

  The final step of the proof is as follows. Write 
  \begin{equation}
    \label{eq:det_idents}
    \begin{aligned}
      \fdet{\op{I} - \op{H}_a} &= \fdet{\op{I} -\op{A}_a\op{B}} = \fdet{
        \begin{pmatrix}
          \op{I}_1 & 0\\
          0 & \op{I}_2
        \end{pmatrix}
        -
        \begin{pmatrix}
          \op{A}_a\op{B} & 0\\
          0 & 0
        \end{pmatrix}
      }\\
      &= \fdet{
        \begin{pmatrix}
          \op{I}_1 & 0\\
          0 & \op{I}_2
        \end{pmatrix}
        -
        \begin{pmatrix}
          \op{A}_a\op{B} & 0\\
          \op{B} & 0
        \end{pmatrix}
      }\\
      &=\fdet{\left(
          \begin{pmatrix}
            \op{I}_1 & 0\\
            0 & \op{I}_2
          \end{pmatrix}
          +
          \begin{pmatrix}
            0 & \op{A}_a\\
            0 & 0
          \end{pmatrix}
        \right)\left(
          \begin{pmatrix}
            \op{I}_1 & 0\\
            0 & \op{I}_2
          \end{pmatrix}
          -
          \begin{pmatrix}
            0 & \op{A}_a\\
            \op{B} & 0
          \end{pmatrix}
        \right)
      }\\
      &= \fdet{
        \begin{pmatrix}
          \op{I}_1 & 0\\
          0 & \op{I}_2
        \end{pmatrix}
        +
        \begin{pmatrix}
          0 & \op{A}_a\\
          0 & 0
        \end{pmatrix}
      }\fdet{
        \begin{pmatrix}
          \op{I}_1 & 0\\
          0 & {I}_2
        \end{pmatrix}
        -
        \begin{pmatrix}
          0 & \op{A}_a\\
          \op{B} & 0
        \end{pmatrix}
      }\\
      &=\fdet{
        \begin{pmatrix}
          \op{I}_1 & 0\\
          0 & \op{I}_2
        \end{pmatrix}
        -
        \begin{pmatrix}
          0 & \op{A}_a\\
          \op{B} & 0
        \end{pmatrix}
      },
    \end{aligned}
  \end{equation}
  where~$\op{I}_1$ and~$\op{I}_2$ are identity operators on~$L_2(\tilde{\gamma})$ and~$L_2(\gamma)$, respectively. The identities above  make sense since all involved  operators  are trace-class, as proved earlier.

  Since~$\gamma \cap \tilde{\gamma} = \varnothing$, we can use an isometric isomorphism between~$L_2(\tilde{\gamma}) \oplus L_2(\gamma)$ and~$L_2(\tilde{\gamma} \cup \gamma)$ 
  \begin{equation}
    f_1 \oplus f_2 \mapsto {1}_{\tilde{\gamma}} \tilde{f}_1 
    + {1}_{\gamma}\tilde{f}_2, 
  \end{equation}
  where~$\tilde{f}_1$ and~$\tilde{f}_2$ are the extensions of~$f_1$ and~$f_2$ to the domain~$\tilde{\gamma} \cup \gamma$ by zero.

  For the operator in the last determinant~\eqref{eq:det_idents}, this isomorphism yields
  \begin{equation}
    \begin{pmatrix}
      0 & \op{A}_a\\
      \op{B} & 0
    \end{pmatrix}
    \mapsto \chi_{\tilde{\gamma}} \widetilde{\op{A}} + \chi_{\gamma} \widetilde{\op{B}} \overset{\mbox{def}}{=}\op{Q}_a,
  \end{equation}
  where~$\widetilde{A}$ and~$\widetilde{B}$ are extensions of~$\op{A}$ and~$\op{B}$ to~$L_2(\tilde{\gamma} \cup \gamma)$ by zero, and where~$\chi_{\gamma}$ and~$\chi_{\tilde{\gamma}}$ are the multiplication operators by the indicators~$1_{\gamma}$ and~$1_{\tilde{\gamma}}$.

  The kernel of~$\op{Q}_a$ is
  \begin{equation}
    Q_a(x,y) = A_a(x,y) {1}_{\tilde{\gamma}}(x) {1}_{\gamma}(y) + B(x,y) {1}_{\tilde{\gamma}}(y) {1}_{\gamma}(x),
  \end{equation}
  and~$A_a(x,y)$ and~$B(x,y)$ are given by~\eqref{63a} and \eqref{63b}.

  We end up with the formula~\eqref{KernelQa} for~$Q(x,y)$. This concludes the proof of the theorem.
\end{proof}


\section{Gap probability in terms of the Riemann--Hilbert problem}
\label{Section7}
In this section we represent the gap probability~$\mathcal{P}(a)$ in terms of the solution of~\ref{problem_y-rh}. Due to our Theorem~\ref{TheoremGapProbabilityIntegrableOperator}, we know that~$\mathcal{P}(a)$ can be written as the Fredholm determinant of the operator~$\op{Q}_a$, so it remains to show that this determinant can be laid out in terms of the solution of~\ref{problem_y-rh}.

We start off by proving an auxiliary lemma.
\begin{lem}
  \label{Lemma71}
  There exists some constant~$C>0$ such that
  \begin{equation}
    \label{In72}
    |1-\mathcal{P}(a)| = \left|1-\fdet{\op{I}-\op{K}|_{L_2\left(a,+\infty\right)}}\right|\leq C e^{-\frac{a}{2}},\quad a>0.
  \end{equation}

  Moreover, for all~$a \in \mathbb{R}$ the gap probability~$\mathcal{P}(a)$ is strictly positive and~$(\op{I}-\op{Q}_a)$ has an inverse.
\end{lem}
\begin{proof}First, we note that
  \begin{equation}
    \label{eq72}
    \fdet{\op{I}-\op{K}|_{L_2\left(a,+\infty\right)}}  =\fdet[\nu]{\op{I}-\widetilde{\op{K}}|_{L_2\left(a,+\infty\right)}},
  \end{equation}
  where~$\widetilde{K}(x,y)=K(x,y)e^{\frac{x+y}{4}}$, and the Fredholm determinant on the right-hand side of~\eqref{eq72} is defined with respect to the measure~$\nu(dx)=e^{-x/2}\, dx$.

  Next, Lemma~3.4.5 from Anderson, Guionnet, and Zeitouni~\cite{AndersonGuionnetZeitouni} yields
  \begin{equation}
    \label{eq:ineq1}
    \left|1-\fdet{\op{I}-\widetilde{\op{K}}|_{L_2\left(a,+\infty\right)}}\right|\leq\left(\sum\limits_{n=1}^{\infty}
      \frac{n^{\frac{n+1}{2}}\left(\|\nu\|_1\right)^n\left(\left\|\widetilde{\op{K}}|_{L_2\left(a,+\infty\right)}\right\|_{\infty}\right)^{n-1}}{n!}\right)\left\|\widetilde{\op{K}}|_{L_2\left(a,+\infty\right)}\right\|_{\infty},
  \end{equation}
  where~$\left\|\widetilde{\op{K}}|_{L_2\left(a,+\infty\right)}\right\|_{\infty}=\underset{x,y>a}{\sup} \left| \widetilde{K}(x,y)\right|$ and~$\|\nu\|_1=\int\limits_a^{+\infty}|\nu(dx)|=2e^{-\frac{a}{2}}$.
  
  Take~$a = 0$ in Proposition~\ref{PropositionTraceClass}. The bound follows,
  \begin{equation}
    \label{eq:ineq2}
    \left\|\widetilde{\op{K}}|_{L_2\left(a,+\infty\right)} \right\|_{\infty} \le C e^{-a/2}, \quad a>0,
  \end{equation}
  where~$C>0$ is an absolute constant. It is readily verified that
  \begin{equation}
    \label{eq:ineq3}
    \sum\limits_{n=1}^{\infty} \frac{n^{\frac{n+1}{2}}\left(\|\nu\|_1\right)^n\left(\left\|\widetilde{\op{K}}|_{L_2\left(a,+\infty\right)}\right\|_{\infty}\right)^{n-1}}{n!}\le \widetilde{C},\quad a>0,
  \end{equation}
  where~$\widetilde{C}>0$ is an absolute constant. Then, the inequalities~\eqref{eq:ineq1}, \eqref{eq:ineq2}, and~\eqref{eq:ineq3} prove~\eqref{In72}. 

  Due to Remark~\ref{Q_a_analytic} and the fact that~\eqref{KernelQa} is analytic in~$a$, the determinant~$\fdet{\op{I}-\op{Q}_a}$ is an entire function of~$a \in \mathbb{C}$. The uniqueness theorem for analytic functions, the fact that~$\fdet{\op{I}-\op{Q}_a} = \mathcal{P}(a)$, $a \in \mathbb{R}$, and the observation that~\eqref{In72} implies~$\mathcal{P}(a) > 0$ for large enough~$a>0$ show that~$\mathcal{P}(a)$ has at most a finite number of zeros in~$[a_0,+\infty)$ for arbitrary~$a_0 \in \mathbb{R}$. It is also clear that, being a distribution function (see the discussion in Section~\ref{SectionDiscussionDistributionFunction}), $\mathcal{P}(a)$ is non-negative and weakly increasing, so its zero with the largest coordinate can only be at~$a=a_0$. Since~$a_0$ is arbitrary, we conclude that~$\mathcal{P}(a) = \fdet{\op{I}-\op{Q}_a} >0$ for all~$a \in \mathbb{R}$. In particular, the determinant is non-vanishing and thus $(\op{I}-\op{Q}_a)$ has an inverse.
\end{proof}

Now, we can establish the following statement.
\begin{prop}
  \label{prop_72}
  The Riemann--Hilbert problem, ~\ref{problem_y-rh}, has a unique solution~$Y(z)$. Let~$Y_1(a)$ be defined by~\eqref{Y1}. Then, $Y_1(a)$ is smooth and
  \begin{equation}
    \label{PP}
    \frac{\mathcal{P}'(a)}{\mathcal{P}(a)}=\left(Y_1(a)\right)_{1,1}.
  \end{equation}
\end{prop}
\begin{proof}
  Set
  \begin{equation}
    \label{f73}
    f(z)=
    \begin{pmatrix}
      f_1(z)\\
      f_2(z) 
    \end{pmatrix}=\frac{1}{2\pi i}
    \begin{pmatrix}
      {1}_{\tilde{\gamma}}(z)e^{\frac{\alpha z^2}{4}-az} \\
      {1}_{\gamma}(z)e^{-\frac{\alpha z^2}{4}} 
    \end{pmatrix},
  \end{equation}
  and
  \begin{equation}\label{h73a}
    h(z)=
    \begin{pmatrix}
      h_1(z)\\
      h_2(z) 
    \end{pmatrix}=
    \begin{pmatrix}
      {1}_{\gamma}(z)\Gamma(z) e^{-\frac{\alpha z^2}{4}+az}\\
      -{1}_{\tilde{\gamma}}(z) \left(\Gamma(z)\right)^{-1} e^{\frac{\alpha z^2}{4}} 
    \end{pmatrix}.
  \end{equation}
  Observe that in these terms, the kernel~$Q_a(x,y)$ in~\eqref{KernelQa} reads
  \begin{equation}
    Q_a(x,y)=\frac{f_1(x)h_1(y)+f_2(x)h_2(y)}{x-y}.
  \end{equation}
  Also, clearly
  \begin{equation}
    f_1(x)h_1(x)+f_2(x)h_2(x)=0,
  \end{equation}
  so~$Q_a(x,y)$ is continuous.

  Since~$(\op{I}-\op{Q}_a)$ has an inverse by Lemma~\ref{Lemma71}, we can introduce a vector function~$F(z)$ whose components~$F_j(z)$ are defined by
  \begin{equation}
    F_j=\left(\op{I}-\op{Q}_a\right)^{-1}[f_j],\quad j=1,2.
  \end{equation}
  It is well-known (e.g., see Its~\cite[Section 9.4.1]{Its}) that~$F(z)$ can be also represented in terms of the solution of the following~$2\times 2$ Riemann--Hilbert problem
  \begin{itemize}
  \item $Y(z)$ is analytic in~$\mathbb{C} \setminus (\gamma \cup \tilde{\gamma})$;
  \item $Y^+(z) = Y^-(z) J_Y(z),\; z \in \gamma \cup \tilde{\gamma}$, $J_Y(z)=I-2\pi i f(z) h^{T}(z)$;
  \item $Y(z)\to  I$, as~$z\to \infty$, $z\in\C\setminus(\gamma\cup\tilde{\gamma})$;
  \end{itemize}  
  by the formula
  \begin{equation}
    F(z)=Y^+(z)f(z).
  \end{equation}

  Writing out~$J_{Y}(z)$, we see that the Riemann--Hilbert problem stated above is identical to~\ref{problem_y-rh}. This problem has a unique solution if and only if~$(\op{I}-\op{Q}_a)$ is invertible (e.g., see Baik, Deift, and Suidan ~\cite[Theorem 5.2.1]{BaikDeiftSuidan}); the latter has already been established.

  It is commonly known that the solution~$Y(z)$ of~\eqref{problem_y-rh} is given via the Cauchy-type integral (e.g., see Its~\cite[Section 9.4.1]{Its} or Deift~\cite{Deift}),
  \begin{equation}
    Y(z)=I-\int\limits_{\gamma\cup\tilde{\gamma}} \frac{F(s)h^{T}(s)}{s-z}\, ds.
  \end{equation}
  From which it is easy to extract the coefficient~$Y_1(a)$ in the expansion~\eqref{Y1},
  \begin{equation}
    \label{eq74}
    Y_1(a)=\int\limits_{\gamma\cup\tilde{\gamma}}F(s)h^{T}(s)ds.
  \end{equation}
  
  Also, we can use the invertibility of~$\left(\op{I}-\op{Q}_a\right)$ and the smoothness of~$Q_a(x,y)$ with respect to~$a$ to conclude that~$Y_1(a)$ is smooth and to obtain
  \begin{equation}
    \label{eq75}
    \frac{\mathcal{P}'(a)}{\mathcal{P}(a)}=-\Tr\left(\left(\op{I}-\op{Q}_a\right)^{-1}\frac{d}{da}\op{Q}_a\right).
  \end{equation}
  It is immediate to see from~\eqref{KernelQa} that
  \begin{equation}
    \frac{d}{da}Q_a(x,y)=-f_1(x)h_1(y),
  \end{equation}
  which gives
  \begin{equation}
    \label{eq76}
    \Tr\left(\left(\op{I}-\op{Q}_a\right)^{-1}\frac{d}{da}\op{Q}_a\right)=-\int\limits_{\gamma\cup\tilde{\gamma}} F_1(s)h_1(s)ds.
  \end{equation}
  The formulas~\eqref{eq74}--\eqref{eq76} yield~\eqref{PP}, and the proof is completed.
\end{proof}

For the purpose of the subsequent asymptotic analysis, it is convenient to relate~$\left(Y_1(a)\right)_{1,1}$ to the off-diagonal elements
of~$Y_1(a)$.
\begin{prop}
  \label{prop_73}
  Let~$Y_1(a)$ be defined via~\eqref{Y1} in terms of the solution~$Y(z)$ of~\ref{problem_y-rh}. Then,
  \begin{equation}
    \label{71part2}
    \frac{d}{da}\left(Y_1(a)\right)_{1,1}=\left(Y_1(a)\right)_{1,2}\left(Y_1(a)\right)_{2,1}.
  \end{equation}
\end{prop}
\begin{proof} 
  Let~$Y(z;a)$ be the solution of~\ref{problem_y-rh}. For the sake of transparency, we are  explicitly writing the parameter~$a$ in the list of arguments of~$Y(z)$. Set
  \begin{equation}
    \Psi(z;a)=Y(z;a)e^{-\frac{az}{2}\sigma_3},\quad \sigma_3=
    \begin{pmatrix}
      1 & 0 \\
      0 & -1
    \end{pmatrix}.
  \end{equation}
  The jump matrix of~$\Psi(z,a)$,
  \begin{equation}
    J_{\Psi}(z)=
    \begin{pmatrix}
      1 & {1}_{\tilde{\gamma}}(z)\left(\Gamma(z)\right)^{-1}e^{\frac{\alpha z^2}{2}} \\
      -{1}_{\gamma}(z)\Gamma(z)e^{-\frac{\alpha z^2}{2}} & 1
    \end{pmatrix},
  \end{equation}
  does not depend on~$a$. Therefore,
  \begin{equation}
    \label{VDPSI}
    V(z;a)=\left(\frac{\partial}{\partial a}\Psi(z;a)\right)\left(\Psi(z;a)\right)^{-1}.
  \end{equation}
  has  no jump in the complex~$z$-plane and thus is entire. The formula~\eqref{Y1} yields the following asymptotic expansion for~$\Psi(z;a)$,
  \begin{equation}
    \Psi(z;a)=\left(I+\frac{Y_1(a)}{z} + \frac{Y_2(a)}{z^2}+O\left(\frac{1}{z^3}\right)\right)e^{-\frac{az}{2}\sigma_3}
  \end{equation}
  as~$z\to \infty$. This gives
  \begin{equation}
    \label{PSI1}
    \begin{aligned}
      &\left(\frac{\partial}{\partial a}\Psi(z;a)\right)\left(\Psi(z;a)\right)^{-1}\\
      &=\left(\frac{Y_1'(a)}{z}+\frac{Y_2'(a)}{z^2}+O\left(\frac{1}{z^3}\right)\right)
      \left(I+\frac{Y_1(a)}{z}+\frac{Y_2(a)}{z^2}+O\left(\frac{1}{z^3}\right)\right)^{-1}\\
      &+\left(I+\frac{Y_1(a)}{z}+\frac{Y_2(a)}{z^2}+O\left(\frac{1}{z^3}\right)\right)\left(-\frac{z}{2}\sigma_3\right)
      \left(I+\frac{Y_1(a)}{z}+\frac{Y_2(a)}{z^2}+O\left(\frac{1}{z^3}\right)\right)^{-1}
    \end{aligned}
  \end{equation}
  as~$z\to\infty$.

  Define the matrices~$A(a)$, $B(a)$, and~$C(a)$ to be the coefficients in the expansion~\eqref{PSI1}, i.e.,
  \begin{equation}\label{PSI2}
    \left(\frac{\partial}{\partial a}\Psi(z;a)\right)\left(\Psi(z;a)\right)^{-1}=A(a) z + B(a) +\frac{C(a)}{z}+O\left(\frac{1}{z^2}\right)
  \end{equation}
  as~$z\to\infty$. Comparing~\eqref{PSI1} and~\eqref{PSI2}, we find
  \begin{equation}
    A(a)=-\frac{1}{2}\sigma_3,\quad B(a)=\frac{1}{2}\left[\sigma_3, Y_1(a)\right],
  \end{equation}
  and
  \begin{equation}
    C(a)=Y_1'(a)+\frac{1}{2}
    \left[\sigma_3,Y_2(a)\right]-\frac{1}{2}\left[\sigma_3,Y_1(a)\right] Y_1(a).
  \end{equation}
  Recall that~$V(z;a)$ is entire in the complex $z$-plane. Liouville's theorem yields
  \begin{equation}
    V(z;a) = A(a) z + B(a),
  \end{equation}
  so~$C=0$, and we find out that
  \begin{equation}
    \label{eq725}
    Y_1'(a)+\frac{1}{2} \left[\sigma_3,Y_2(a)\right]-\frac{1}{2}\left[\sigma_3,Y_1(a)\right] Y_1(a)=0.
  \end{equation}
  Extracting the~$(1,1)$-elements of the left and right hand sides of~\eqref{eq725}, we arrive at~\eqref{71part2}.
\end{proof}


\section{Asymptotic analysis}
\label{Section8}
By Proposition~\ref{prop_72} and Proposition~\ref{prop_73}, the logarithmic derivative of the gap probability~$\mathcal{P}(a)$ is determined by the product of the off-diagonal elements of~$Y_1(a)$. Recall that the matrix~$Y_1(a)$ is a coefficient before~$z^{-1}$ in the large z expansion of the unique solution~$Y(z)$ of~\ref{problem_y-rh}.

In order to find the asymptotics of the product~$\left(Y_1(a)\right)_{1,2} \left(Y_1(a)\right)_{2,1}$ as~$a\to +\infty$, we need to transform our original problem, \ref{problem_y-rh}. Since we study the asymptotics as~$a \to +\infty$, we assume that~$a>0$ in this section.  

To begin with, we switch the direction on~$\tilde{\gamma}$ and deform this contour in such a way that it passes through~$a/\alpha$. Then, we deform~$\gamma$ so that it intersects the $x$-axis at~$1/(\alpha a)$. As a result, we obtain a new Riemann--Hilbert problem, which can be stated in a similar way as~\ref{problem_y-rh}; the only difference is that its jump matrix,
\begin{equation}
  \begin{aligned}
    \widetilde{J}_Y(z)= 
    \begin{pmatrix}
      1 &  -1_{\tilde{\gamma}}(z)(\Gamma{(z)})^{-1}e^{\frac{\alpha z^2}{2} - a z}\\
      - 1_{\gamma}(z) \Gamma{(z)} e^{-\frac{\alpha z^2}{2} + a z}  & 1
    \end{pmatrix},
  \end{aligned}
\end{equation}
has an opposite sign in the~$(1,2)$ entry. It is not hard to check that these problems are equivalent.

Next, consider the transformation
\begin{equation}
  \label{Transformation82}
  \zeta(z)=i\left(\frac{\alpha z}{a}-1\right)
\end{equation}
and its inverse
\begin{equation}
  \label{Transformation83}
  z(\zeta)=\frac{a}{\alpha}\left(1-i\zeta\right).
\end{equation}
The images~$\lambda$ and~$\tilde{\lambda}$ of~$\gamma$ and~$\tilde{\gamma}$ under~\eqref{Transformation82} are shown in Fig.~\ref{norm_cont}.
\begin{figure}[ht!] 
  \centering
  \begin{tikzpicture}[scale=2.5]
    \begin{scope}[compass style/.style={color=black}, color=black, decoration={markings,mark= at position 0.5 with {\arrow{stealth}}}]
      \draw[->, thick] (-1,0) -- (1,0);
      \draw[->] (0,-1.4) -- (0,0.5);

      \draw[thick, postaction=decorate, rotate=90, xshift=-0.5cm, yscale=1.5, xscale=0.75] (-1.05,-0.25) -- (0,-0.25);
      \draw[thick, postaction=decorate, rotate=90, xshift=-0.5cm, yscale=1.5, xscale=0.75] (0,0.25) -- (-1.05,0.25);
      \draw[thick, postaction=decorate, rotate=90, xshift=-0.5cm, yscale=1.5] (0,-0.25)  arc[start angle=-90, end angle=90,radius=0.25];

      \node at (-0.05,-0.1) {0};
      \node at (0.58,+0.1) {$\tilde{\lambda}$};
      \node at (-0.5,-0.3) {$\lambda$};

      \node[circle,fill, inner sep=1.0pt] () at (0,0){};
      
      \draw (0,-0.5) node [cross] {};
      \node at (0.07,-0.15) {$\theta$};
      \node at (0.2,-0.52) {$-i$};
      \node at (0.2,-1.02) {$-2i$};
      \node at (-0.7,0.4) {$\zeta$-plane};
      
      \draw (0,-1) node [cross] {};
    \end{scope}
  \end{tikzpicture}  
  \caption{The contours~$\lambda$ and~$\tilde{\lambda}$ in the $\zeta$-plane. The contour~$\lambda$ intersects the y-axis at~$\theta = -i(1-1/a^2)$.}
  \label{norm_cont}
\end{figure}
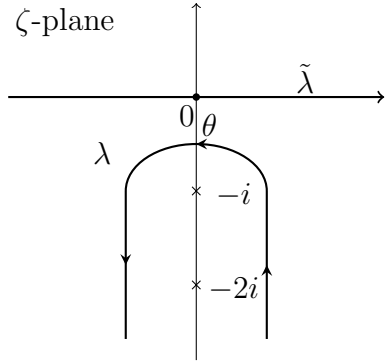

We normalize our Riemann--Hilbert problem by defining a new matrix function~$U(\zeta)$,
\begin{equation}
  \label{eq:small_norm_u_transform}
  U(\zeta) = e^{\frac{a^2}{8 \alpha} \sigma_3} Y\left(\frac{a}{\alpha}(1-i \zeta)\right)e^{-\frac{a^2}{8 \alpha} \sigma_3},\quad \zeta \in \mathbb{C} \setminus(\lambda \cup \tilde{\lambda}).
\end{equation}
This function satisfies the following Riemann--Hilbert problem.
\begin{problem}[U-RH]
  \needspace{5ex}
  \namedlabel{problem_u-rh}{Problem U-RH}
  \leavevmode
  \begin{enumerate}[label=\textnormal{({\roman*})},ref=U-RH-{\roman*}]
  \item  \label{urh_cond1} $U(\zeta)$ is analytic in~$\mathbb{C} \setminus (\lambda \cup \tilde{\lambda})$;
  \item  \label{urh_cond2} $U^+(\zeta) = U^-(\zeta) J_U(\zeta),\quad \zeta \in \lambda \cup \tilde{\lambda}$,
    \begin{equation}
      \label{eq_jump_U}
      J_U(\zeta) =
      \begin{pmatrix}
        1 &  -{1}_{\tilde{\lambda}}(\zeta)\left(\Gamma{\left(\frac{a}{\alpha}(1- i \zeta)\right)}\right)^{-1} e^{-\frac{a^2}{2 \alpha}(\zeta^2+\frac{1}{2})}\\
        - {1}_{\lambda}(\zeta) \Gamma{\left(\frac{a}{\alpha}(1- i \zeta)\right)} e^{\frac{a^2}{2 \alpha}(\zeta^2+\frac{1}{2})}  & 1
      \end{pmatrix},
    \end{equation}
  \item  \label{urh_cond3} $U(\zeta) \to I$ as~$\zeta \to \infty$, $\zeta \in \mathbb{C} \setminus (\lambda \cup \tilde{\lambda})$.
  \end{enumerate}
\end{problem}

Equations~\eqref{eq:small_norm_u_transform}, \eqref{Transformation82}, and~\eqref{Y1} show that
\begin{equation}
  \label{Relation87}
  U_1(a)=\frac{i\alpha}{a}e^{\frac{a^2}{8\alpha}\sigma_3}Y_1(a) e^{-\frac{a^2}{8\alpha}\sigma_3},
\end{equation}
where~$U_1(a)$ is the coefficient before~$\zeta^{-1}$ in the large~$\zeta$ expansion of~$U(\zeta)$
\begin{equation}\label{U86}
  U(\zeta)=I+\frac{U_1(a)}{\zeta}+O\left(\frac{1}{\zeta^2}\right), \quad \zeta \to \infty.
\end{equation}

In particular, the formula~\eqref{Relation87} implies
\begin{equation}
  \label{Rl88}
  \left(Y_1(a)\right)_{1,2}\left(Y_1(a)\right)_{2,1}=-\frac{a^2}{\alpha^2}
  \left(U_1(a)\right)_{1,2}\left(U_1(a)\right)_{2,1}.
\end{equation}
Therefore, the asymptotic analysis of~$Y_1(a)$ as~$a \to +\infty$ is reduced to that of~$U_1(a)$.

\begin{prop}
  \label{th_RH_U_assympt}
  Let~$U(\zeta)$ be a solution of~\ref{problem_u-rh}, and let~$U_1(a)$ be defined by~\eqref{U86}. Then the following asymptotic identities hold
  \begin{equation}
    \label{eq:U_asymp_RH}
    \begin{aligned}
      (U_1(a))_{1,1} &= O\left(e^{-\frac{a^2}{2 \alpha}}\right), \quad (U_1(a))_{2,2} = O\left(e^{-\frac{a^2}{2 \alpha}}\right),\\
      (U_1(a))_{2,1} &= \frac{1}{2 \pi i} \int \limits_{\lambda} \Gamma{\left(\frac{a}{\alpha}(1- i s)\right)} e^{\frac{a^2}{2 \alpha}\left(s^2+\frac{1}{2}\right)}\, ds + O\left(e^{-\frac{a^2}{2 \alpha}}\right), \\
      (U_1(a))_{1,2} &= \frac{1}{2 \pi i} \int \limits_{-\infty}^{+\infty} \left(
        \Gamma{\left(\frac{a}{\alpha}(1- i s)\right)} \right)^{-1} e^{-\frac{a^2}{2 \alpha}\left(s^2+\frac{1}{2}\right)}\, ds + O\left(e^{-\frac{a^2}{2 \alpha}}\right),
    \end{aligned}
  \end{equation}
  as~$a \to +\infty$.
\end{prop}
\begin{proof}
  Recall the following inequality for the gamma function (see~\cite[Formula (5.6.6), (5.6.7)]{DLMF}),
  \begin{equation}
    \label{eq:est_gamma_inv}
    |(\Gamma{(x + i y)})^{-1}| \le   \frac{\sqrt{\cosh{(\pi y)}}}{{\Gamma}{(x)}}, \quad x \ge 1/2. 
  \end{equation}

  For~$\zeta \in \tilde{\lambda}$, we have
  \begin{equation}
    \left|\left(\Gamma{\left(\frac{a}{\alpha}(1- i\zeta)\right)}\right)^{-1} e^{-\frac{a^2}{2 \alpha}\left(\zeta^2+\frac{1}{2}\right)}\right| \le \frac{\sqrt{
        \cosh{\left(\frac{\pi a \zeta}{\alpha}\right)}}}{\Gamma{\left(\frac{a}{\alpha}\right)}} e^{-\frac{a^2}{2 \alpha}\left(\zeta^2+\frac{1}{2}\right)}, \quad \zeta \in \tilde{\lambda}.
  \end{equation}
  This yields
  \begin{equation}
    \label{I812}
    \left\|\left(\Gamma{\left(\frac{a}{\alpha}(1- i\zeta)\right)}\right)^{-1} e^{-\frac{a^2}{2 \alpha}\left(\zeta^2+\frac{1}{2}\right)}\right\|_{L_p(\tilde{\lambda})} \le C 
    \left\|\sqrt{\cosh{(\pi \zeta)}} e^{-\frac{{\alpha \zeta^2}}{2}}\right\|_{L_p(\tilde{\lambda})} \frac{e^{-\frac{a^2}{4 \alpha}}}{\Gamma{\left(\frac{a}{\alpha}\right)}}, \quad p \in [1,\infty],
  \end{equation}
  for some constant~$C>0$ independent of~$a$.

  For~$\zeta \in \lambda$, we have
  \begin{equation}
    \label{I813}
    \left\|\Gamma{\left(\frac{a}{\alpha}(1- i\zeta)\right)} e^{\frac{a^2}{2 \alpha}\left(\zeta^2+\frac{1}{2}\right)}\right\|_{L_p(\lambda)} \le C \left\|\Gamma(z) e^{-\frac{\alpha z^2}{2}}\right\|_{L_p(\gamma)} e^{-\frac{a^2}{4 \alpha}}, \quad p \in [1,\infty],
  \end{equation}
  for~$C>0$ which does not depend on~$a$. The later inequality is easy to prove by writing out the norm on the left-hand side and changing the variables according to~\eqref{Transformation82}. We remind that the contour~$\gamma$ is specified in Fig.~\ref{contours} and it intersects the $x$-axis at~$1/(\alpha a)$.

  Set
  \begin{equation}
    \Delta = J_U - I,
  \end{equation}
  then the inequalities~\eqref{I812} and~\eqref{I813} imply
  \begin{equation}
    \label{eq:proof_U_int0}
    \|\Delta (\zeta)\|_{L_p(\lambda \cup \tilde{\lambda})} \le C e^{-\frac{a^2}{4 \alpha}}, \quad p \in [1, \infty].
  \end{equation}
  By the standard theory of small-norm Riemann--Hilbert problems for varying contours this means that~\ref{problem_u-rh} has a unique solution for large enough~$a$ and that~$U^-(\zeta)$ satisfies
  \begin{equation}
    \label{eq:proof_U_mu}
    \|U^-(\zeta)-I\|_{L_2(\lambda \cup \tilde{\lambda})} \le C e^{-\frac{a^2}{4 \alpha}}.
  \end{equation}

  For the reader's convenience, we recall essential points of this somewhat standard argument. In that we mostly follow Its~\cite[Section 9.3]{Its}, Fokas, Its, Kapaev, and Novokshenov~\cite[Chapter 3, \S 1]{FokasItsKapaevNovokshenov}, and Deift~\cite{Deift}.

  Let~$\op{C}_{\pm}$ denote the Cauchy operators corresponding to the contour~$\lambda\cup\tilde{\lambda}$, i.e.,
  \begin{equation}
    \left(\op{C}_{\pm}[f]\right)(\zeta)=\frac{1}{2\pi i} \lim \limits_{\xi\to\zeta^{\pm}}
    \int\limits_{\lambda\cup\tilde{\lambda}}\frac{f(\tau)}{\tau-\xi}\, d\tau,\quad \zeta \in \lambda\cup\tilde{\lambda},
  \end{equation}
  where the limits are taken from the~$(\pm)$-side of~$\lambda\cup\tilde{\lambda}$. The solution of the Riemann--Hilbert problem  admits the following integral representation
  \begin{equation}
    \label{815}
    U(\zeta)=I+\frac{1}{2\pi i}\int\limits_{\lambda\cup\tilde{\lambda}}\frac{m(\tau) \Delta(\tau)}{\tau-\zeta}\, d\tau,\quad \zeta\notin\lambda\cup\tilde{\lambda},
  \end{equation}
  where~$m = U^{-}$ solves the integral equation
  \begin{equation}
    \label{816}
    m=\op{I}+\op{C}_{-}\left[m\Delta\right],\quad \zeta\in\lambda\cup\tilde{\lambda}.
  \end{equation}

  Now, set
  \begin{equation}
    m_0 = m-I.
  \end{equation}
  The equation~\eqref{816} takes the form
  \begin{equation}
    \label{819}
    m_0=\op{C}_{-}\left[m_0\Delta\right]+\op{C}_{-}[\Delta],\quad \zeta\in\lambda\cup\tilde{\lambda}.
  \end{equation}
  Next, we would like to pass to the norms in this equation. Using an important property of the boundedness of~$\op{C}_-$ (e.g., see Fokas, Its, Kapaev, and Novokshenov~\cite[Chapter 3, \S1]{FokasItsKapaevNovokshenov}) and the estimate~\eqref{eq:proof_U_int0}, we arrive at
  \begin{equation}
    \label{eq:norm1_est}
    \|C_-[\Delta]\|_{L_2\left(\lambda\cup\tilde{\lambda}\right)} \le \|C_-\|_{\mathrm{op}} \|\Delta\|_{L_2\left(\lambda\cup\tilde{\lambda}\right)}\le C e^{-\frac{a^2}{4\alpha}},
  \end{equation}
  where~$\|\cdot\|_{\mathrm{op}}$ is the corresponding operator norm. Further, set
  \begin{equation}
    \op{C}^{\Delta}_-[f]\overset{\mathrm{def}}{=}C_-\left[f\Delta\right], \quad f \in {L_2\left(\lambda\cup\tilde{\lambda}\right)}.
  \end{equation}
  Using the same reasoning as above, we obtain
  \begin{equation}
    \label{eq:norm2_est}
    \|\op{C}^{\Delta}_-\|_{L_2\left(\lambda\cup\tilde{\lambda}\right)} \le \|\op{C}_-\|_{\mathrm{op}}\|\Delta\|_{L_{\infty} \left(\lambda\cup\tilde{\lambda}\right)}\leq Ce^{-\frac{a^2}{4\alpha}}.
  \end{equation}
  
  Note that the contour~$\lambda$ changes with~$a$; therefore, so do~$\|\op{C}_-\|_{\mathrm{op}}$ and~$C$ in~\eqref{eq:norm1_est} and~\eqref{eq:norm2_est}. Nonetheless, one can choose~$C>0$ so that it is independent of~$a>0$. This is justified because of the following general fact (see Bleher and Kuijlaars~\cite[Appendix A]{BleherKuijlaars} and Coifman, McIntosh, and Meyer~\cite{CoifmanMcIntoshMeyer}).

  Assume that~$\left\{\Gamma_{\sigma}\right\}_{\sigma\in\Lambda}$ is a family of contours in the complex plane each of which is of the parametric form
  \begin{equation}
    \Gamma_{\sigma}=\left\{z \in \mathbb{C}|\, \re{z}=p,\; \im{z}=\varphi_{\sigma}(p),\; -\infty<p<+\infty\right\},
  \end{equation}
  where~$\varphi_{\sigma}(p)$ is uniformly Lipschitz, i.e.,
  \begin{equation}
    \label{eq: Lip_cond}
    |\varphi_{\sigma}(x)-\varphi_{\sigma}(y)|\leq M|x-y|
  \end{equation}
  for some~$M>0$ independent of~$\sigma$. Then, the family of Cauchy operators~$\op{C}_{\pm,\Gamma_{\sigma}}$ associated with the contours~$\Gamma_{\sigma}$ is uniformly bounded, i.e., there
  exists a constant~$C>0$ which only depends on~$M$ and such that for all~$\sigma \in \Lambda$
  \begin{equation}
    \|\op{C}_{\pm,\Gamma_{\sigma}}\|_{\mathrm{op}}\leq C.
  \end{equation}
  
  In our concrete situation, it is clear that in a vicinity of~$\theta = -i\left(1-\frac{1}{a^2}\right)$, the contour~$\lambda$ (see Fig.~\ref{norm_cont}) can be chosen so that the Lipschitz condition~\eqref{eq: Lip_cond} holds uniformly in~$a>0$; the vertical parts of~$\lambda$ can be chosen to be independent of~$a$ at all. This means the premises are satisfied, and thus the operator norm~$\|C_-\|_{\mathrm{op}}$ is bounded uniformly in~$a>0$ as claimed above.

  Now, we get back to~\eqref{819} and~\eqref{eq:norm2_est}. It follows that for large enough~$a>0$ the operator~$\op{C}^{\Delta}_-$ is a contraction. Therefore, the integral equation~\eqref{819} has a unique solution which satisfies the bound (e.g., see Reed and Simon~\cite[Theorem V.21]{ReedSimon})
  \begin{equation}\label{Inequalitym0}
    \|m_0\|_{L_2\left(\lambda\cup\tilde{\lambda}\right)}\leq 2\|C_-[\Delta]\|_{L_2\left(\lambda\cup\tilde{\lambda}\right)}\leq Ce^{-\frac{a^2}{4\alpha}}.
  \end{equation}

  The final part of the proof is as follows. The formulas~\eqref{U86}, \eqref{815}, and~\eqref{816} give us
  \begin{equation}
    \label{824}
    \begin{aligned}
      &U_1(a)=-\frac{1}{2\pi i}\int\limits_{\lambda\cup\tilde{\lambda}}m(\tau)\Delta(\tau)\, d\tau\\
      &=-\frac{1}{2\pi i}\int\limits_{\lambda\cup\tilde{\lambda}}\Delta(\tau)\, d\tau-\frac{1}{2\pi i}\int\limits_{\lambda\cup\tilde{\lambda}}m_0(\tau)\Delta(\tau)\, d\tau.
    \end{aligned}
  \end{equation}
  The second integral on the right-hand side of this expression can be estimated by the Cauchy--Bunyakovsky--Schwarz inequality using~\eqref{eq:proof_U_int0} and~\eqref{Inequalitym0}. Namely, we have
  \begin{equation}
    \label{825}
    \left|\,\int\limits_{\lambda\cup\tilde{\lambda}}m_0(\tau)\Delta(\tau)\, d\tau\right|
    \leq\|m_0\|_{L_2\left(\lambda\cup\tilde{\lambda}\right)}\|\Delta\|_{L_2\left(\lambda\cup\tilde{\lambda}\right)}\leq Ce^{-\frac{a^2}{2\alpha}}.
  \end{equation}
  This inequality together with~\eqref{824} and with~\eqref{eq_jump_U} yield the asymptotic formulas~\eqref{eq:U_asymp_RH}. The proof is concluded.
\end{proof}
\begin{prop}
  \label{Proposition82}
  For the integrals in~\eqref{eq:U_asymp_RH} the following  asymptotic formulas hold,
  \begin{equation}
    \label{As826}
    \frac{1}{2 \pi i} \int \limits_{\lambda} \Gamma{\left(\frac{a}{\alpha}(1- i \zeta)\right)} e^{\frac{a^2 \zeta^2}{2 \alpha}}\, d\zeta=\frac{i\alpha}{a}e^{-\frac{a^2}{2\alpha}}\left(1+O\left(e^{-a}\right)\right),
  \end{equation}
  \begin{equation}
    \label{As827}
    \frac{1}{2 \pi i} \int \limits_{-\infty}^{+\infty} \left(\Gamma{\left(\frac{a}{\alpha}(1- i x)\right)} \right)^{-1} e^{-\frac{a^2 x^2}{2 \alpha}}\, dx =\frac{e^{-\frac{1}{2\alpha} \left(\log\frac{a}{\alpha}\right)^2}}{i\sqrt{2\pi\alpha}\Gamma\left(1+\frac{a}{\alpha}\right)} \left(1+O\left(\frac{(\log a)^2}{a}\right)\right),
  \end{equation}
  as~$a\to+\infty$.
\end{prop}
\begin{proof}
  Let us first obtain~\eqref{As826}. Recall that the contour~$\lambda$ is the image of the contour~$\gamma$ under the transformation~\eqref{Transformation82}. Therefore, we can write
  \begin{equation}
    \label{As828}
    \frac{1}{2 \pi i} \int \limits_{\lambda} \Gamma{\left(\frac{a}{\alpha}(1- i \zeta)\right)} e^{\frac{a^2 \zeta^2}{2 \alpha}}\, d\zeta=\frac{\alpha}{2\pi a}e^{-\frac{a^2}{2\alpha}}\int\limits_{\gamma} \Gamma(z)e^{-\frac{\alpha z^2}{2}+az}\, dz.
  \end{equation}
  The integral along~$\gamma$ on the right-hand side of~\eqref{As828} can be evaluated by the residue theorem,
  \begin{equation}
    \label{As829}
    \frac{1}{2\pi i} \int\limits_{\gamma} \Gamma(z)e^{-\frac{\alpha z^2}{2}+az}\,dz=\sum\limits_{k=0}^{\infty} \frac{(-1)^k}{k!}e^{-\frac{\alpha k^2}{2}-ak},
  \end{equation}
  where we used~$\underset{z=-k}\Res \Gamma(z)=\frac{(-1)^k}{k!}$. The sum above is equal to~$1+O\left(e^{-a}\right)$ as~$a\to+\infty$. Thus, the formula~\eqref{As826} is a direct consequence of~\eqref{As828} and~\eqref{As829}.

  To obtain the asymptotic formula~\eqref{As827}, we denote by~$A(a)$ the integral in its left-hand side, which after massaging takes the form
  \begin{equation}
    \label{As830}
    A(a)=\frac{1}{2\pi i\Gamma\left(\frac{a}{\alpha}\right)}
    \int\limits_{-\infty}^{+\infty} \frac{\Gamma\left(\frac{a}{\alpha}\right)}{\Gamma\left(\frac{a}{\alpha}\left(1-ix\right)\right)} e^{-\frac{a^2 x^2}{2\alpha}}dx.
  \end{equation}
  In order to guess the asymptotics of~$A(a)$ as~$a\to +\infty$, we can use Stirling's formula~\eqref{StirlingFormula} once again. Applying this formula to the ratio of gamma functions in~\eqref{As830} and omitting the low-order terms, we arrive at
  \begin{equation}
    \label{As831}
    A(a)\sim\frac{1}{2\pi i\Gamma\left(\frac{a}{\alpha}\right)}
    \int\limits_{-\infty}^{+\infty} \myexp{-\frac{a^2x^2}{2\alpha}+\frac{i a x}{\alpha}\log\frac{a}{\alpha}}\, dx
  \end{equation}
  as~$a\to+\infty$. The integral above can be computed explicitly, and we obtain
  \begin{equation}
    \label{As832}
    A(a)\sim\frac{e^{-\frac{1}{2\alpha}\left(\log\frac{a}{\alpha}\right)^2}}{
      i\sqrt{2\pi\alpha}\Gamma\left(1+\frac{a}{\alpha}\right)}
  \end{equation}
  as~$a\to+\infty$.

  To establish~\eqref{As827} rigorously, observe that~$A(a)$ can be written as
  \begin{equation}
    A(a)=\frac{e^{-\frac{1}{2\alpha}\left(\log\frac{a}{\alpha}\right)^2}}{
      i\sqrt{2\pi\alpha}\Gamma\left(1+\frac{a}{\alpha}\right)}\left(1+ \sqrt{\frac{\alpha}{2 \pi}}\int \limits_{-\infty}^{+\infty} Q\left(s;\frac{a}{\alpha}\right) \myexp{-\frac{\alpha}{2}\left(s - \frac{i}{\alpha} \log\frac{a}{\alpha}\right)^2}\, ds\right),
  \end{equation}
  where
  \begin{equation}
    \label{eq:def_Q}
    Q(s;t) = \frac{\Gamma{(t)}e^{- i s \log{t}}}{\Gamma{\left(t - i s\right)}} - 1.
  \end{equation}
  We deform the contour~$\mathbb{R} \mapsto \mathbb{R}+ i{\alpha}^{-1} \log{(a/\alpha)}$ and make the change of variables~$s \mapsto s + i{\alpha}^{-1}\log{(a/\alpha)}$, which brings the new contour back into~$\mathbb{R}$. This gives
  \begin{equation}
    A(a)=\frac{e^{-\frac{1}{2\alpha}\left(\log\frac{a}{\alpha}\right)^2}}{
      i\sqrt{2\pi\alpha}\Gamma\left(1+\frac{a}{\alpha}\right)}\left(1+ \sqrt{\frac{\alpha}{2 \pi}}\int \limits_{-\infty}^{+\infty}Q\left(s+i\phi(a);\frac{a}{\alpha}\right)e^{-\frac{\alpha s^2}{2}} ds\right),
  \end{equation}
  where~$\phi(a)=\alpha^{-1}\log{(a/\alpha)}$. The inequality~\eqref{eq:est_gamma_inv} implies the estimate 
  \begin{equation}
    \left|Q\left(s+ {i}\phi(a);\frac{a}{\alpha}\right)\right| \le  1+ 
    \frac{\Gamma\left(\frac{a}{\alpha}\right) \sqrt{\cosh{\pi s}}}{\Gamma\left(\frac{a}{\alpha} + \frac{1}{\alpha}\log\frac{a}{\alpha}\right)} \le 1 + e^{\frac{\pi |s|}{2}} \le 2e^{ \frac{\pi |s|}{2}}.  
  \end{equation}
  Using this estimate we obtain
  \begin{equation}
    \label{eq:est_inf_s1}
    \left|\, \int \limits_{|s|> \log{a}} Q\left(s + {i}\phi(a);\frac{a}{\alpha}\right) e^{-\frac{\alpha s^2}{2}}\, ds\right| \le 2 \int \limits_{|s|> \log{a}} e^{\frac{\pi |s|}{2}} e^{-\frac{\alpha s^2}{2}}\, ds \le C e^{-\varkappa (\log{a})^2},
  \end{equation}
  for some~$C>0$ and~$\varkappa>0$. For~$|s| \le \log{a}$, we note that the arguments of the integral satisfy~$s+i \phi(a)=O(\log{a})$ and~$a/\alpha = O(a)$ as~$a \to +\infty$, both~$O$-terms are uniform in~$s$ for the specified interval. Having this is mind, we calculate asymptotics of~$Q(\tilde{s},{\tilde{a}})$ for~$\tilde{s} = O(\log{a})$ and~$\tilde{a} = O(a)$.
  
  Since both arguments of the gamma function in~\eqref{eq:def_Q} approach infinity,  we can use~\eqref{StirlingFormula} to find out that
  \begin{equation}
    \begin{aligned}
      Q\left(\tilde{s};\tilde{a}\right) &= \myexp{-\left(\tilde{a}-i \tilde{s}-\frac{1}{2}\right) 
        \log\left(1 - \frac{i\tilde{s}}{\tilde{a}}\right)-i \tilde{s}}\left(1+ O\left(\frac{1}{\tilde{a}}\right)\right) - 1,
    \end{aligned}
  \end{equation}
  uniformly in~$s$ for~$|s| \le \log{a}$ as~$a \to +\infty$. Expanding the logarithm in a series up to a quadratic term, we obtain~$Q(\tilde{s},\tilde{a}) = O\left(\frac{(\log{\tilde{a}})^2}{\tilde{a}}\right)$ as~$a \to +\infty$. Plugging in~$\tilde{s}=s + i \phi(a)$ and~$\tilde{a}=a/\alpha$, we arrive at
  \begin{equation}
    Q(s + i \phi(a), a) = O\left(\frac{(\log{a})^2}{a}\right),
  \end{equation}
  uniformly in~$s$ for~$|s| \le \log{a}$, as~$a \to +\infty$. Therefore,
  \begin{equation}
    \label{eq:est_inf_s2}
    \left|\, \int \limits_{|s| \le \log{a}} Q\left(s + {i}\phi(a);\frac{a}{\alpha}\right) e^{-\frac{\alpha s^2}{2}}\, ds\right|\le C \frac{(\log{a})^2}{a},
  \end{equation}
  for some~$C>0$. This concludes the proof.
\end{proof}


\section{Proof of Theorem~\ref{TheoremMainResult} and of Corollary~\ref{cor_main_theorem}}
\label{Section9}
\begin{proof}[Proof of Theorem~\ref{TheoremMainResult}]
By Proposition~\ref{prop_72}, the gap probability~$\mathcal{P}(a)$ satisfies the differential equation~\eqref{PP}. Taking into account~\eqref{71part2}, we obtain
\begin{equation}\label{PT91}
  \frac{d^2}{da^2}\log\mathcal{P}(a)=\left(Y_1(a)\right)_{1,2}\left(Y_1(a)\right)_{2,1}.
\end{equation}

Integrating both the right-hand side and the left-hand side of~\eqref{PT91} yields
\begin{equation}\label{PT92}
  \frac{d}{da}\log\mathcal{P}(a)=-\int\limits_a^{\infty}\left(Y_1(x)\right)_{1,2}\left(Y_1(x)\right)_{2,1}dx.
\end{equation}
We used the fact that the logarithmic derivative of~$\mathcal{P}(a)$ is equal to 
$\left(Y_1(a)\right)_{1,1}$, which vanishes as~$a\to+\infty$, according to~\eqref{Relation87} and~\eqref{eq:U_asymp_RH}. By Lemma~\ref{Lemma71} and by the formula~\eqref{PFredholmFirst}, the logarithm of~$\mathcal{P}(a)$ also vanishes as~$a\to+\infty$. Therefore, equation~\eqref{PT92} implies
\begin{equation}
  \label{PT93}
  \log\mathcal{P}(a)=\int\limits_a^{+\infty}
  \left(\int\limits_x^{\infty}\left(Y_1(s)\right)_{1,2}\left(Y_1(s)\right)_{2,1}ds\right) dx.
\end{equation}

By Fubini's theorem, it is immediate to see that
\begin{equation}
  \label{PT95}
  \log\mathcal{P}(a)=\int\limits_a^{+\infty}(x-a)
  \left(Y_1(x)\right)_{1,2}\left(Y_1(x)\right)_{2,1} dx.
\end{equation}
This gives formulas~\eqref{MainExactFormula} and~\eqref{URH} in the statement of the theorem. The fact that the function~$u(x)$ defined by~\eqref{URH} has the asymptotics~\eqref{Asymptotics38} follows from~\eqref{Rl88} and from Propositions~\ref{th_RH_U_assympt}--\ref{Proposition82}.
\end{proof}

\begin{proof}[Proof of Corollary~\ref{cor_main_theorem}]
  First, notice that
  \begin{equation}
    \begin{aligned}
      &\left(\int\limits_a^{+\infty}(x-a) u_0(x) \, dx  \right)^2 \le \int\limits_a^{+\infty}(x-a) u_0(x) \frac{(\log{x})^2}{x}\, dx \int\limits_a^{+\infty}(y-a) u_0(y) \frac{y}{(\log{y})^2} \, dy\\
      &\le \int\limits_a^{+\infty}(x-a) u_0(x) \frac{(\log{x})^2}{x}\, dx \int\limits_a^{+\infty}  \frac{u_0(y) y^2}{(\log{y})^2} \, dy 
      \le C \int\limits_a^{+\infty}(x-a) u_0(x) \frac{(\log{x})^2}{x}\, dx
    \end{aligned}
  \end{equation}
  for some~$C>0$, where we used the Cauchy--Bunyakovsky--Schwarz inequality and the straight forward estimates.
  
  Expanding the exponential in~\eqref{MainExactFormula} in a series up to the second order terms and plugging in the asymptotics~\eqref{Asymptotics38} then yield~\eqref{eq:P_asymp}. The formula~\eqref{eq:P_asymp_crude} follows from~\eqref{eq:P_asymp} if one notices that
  \begin{equation}
    \begin{aligned}
      &(u_0(a))^{-1}\int\limits_a^{+\infty}(x-a) u_0(x) \, dx = (u_0(a))^{-1}\int\limits_0^{+\infty} x\, u_0(x+a) \; dx \\
      &=\int\limits_0^{+\infty} \frac{\alpha \Gamma{(1+\frac{x}{\alpha})}}{\Gamma{(\frac{x+a}{\alpha})}} {\myexp{-\frac{1}{2 \alpha}\left((x+a)^2 - a^2+ \left(\log{\frac{x+a}{\alpha}}\right)^2 - \left(\log{\frac{a}{\alpha}}\right)^2\right)}} \, dx \\
      &\le C \int\limits_0^{+\infty}  \Gamma{\left(1+\frac{x}{\alpha}\right)} e^{-\frac{x^2}{2 \alpha}}\, dx < +\infty.
    \end{aligned}
  \end{equation}
\end{proof}


\end{document}